\documentclass[12pt,numbers]{elsarticle}
\journal{}

\makeatletter
\def\ps@pprintTitle{%
 \let\@oddhead\@empty
 \let\@evenhead\@empty
 \def\@oddfoot{\hfill\thepage}%
 \let\@evenfoot\@oddfoot}
\makeatother

\usepackage{amsmath,amssymb,amsfonts,amsthm,graphicx}
\usepackage[bookmarksnumbered,colorlinks=true]{hyperref}
\usepackage[labelfont=bf]{caption}

\providecommand{\doi}[1]{\href{https://doi.org/#1}{DOI:#1}}
\usepackage{xurl} 
\renewcommand{\doi}[1]{%
 \href{https://doi.org/#1}{\nolinkurl{DOI:#1}}%
}

\usepackage{dsfont} 
\usepackage{enumitem} 
\usepackage{mathtools} 
\usepackage{appendix} 
\usepackage{geometry} 
\numberwithin{equation}{section}
\numberwithin{table}{section}
\numberwithin{figure}{section}

\theoremstyle{plain}
\newtheorem{theorem}{Theorem}[section]

\newtheorem{lemma}[theorem]{Lemma}
\newtheorem{corollary}[theorem]{Corollary}

\theoremstyle{definition}
\newtheorem{remark}{Remark}[section]

\newcommand{\Z}{\mathbb{Z}}

\newcommand{\R}{\mathbb{R}}
\newcommand{\PP}{\mathsf{P}} 
\newcommand{\EE}{\mathsf{E}} 
\newcommand{\bb}[1]{\boldsymbol{#1}}
\newcommand{\rd}{\mathrm{d}}

\newcommand{\leqdef}{\vcentcolon=}

\newcommand{\Med}{\operatorname{Med}}
\newcommand{\GIG}{\operatorname{GIG}}

\allowdisplaybreaks

\begin{document}

\begin{frontmatter}

\title{Bounds for the median of the generalized hyperbolic and related distributions \\[-1mm] {\small \itshape Dedicated to the memory of Milan Merkle}\vspace{-1mm}}

\author[a2]{Robert E.~Gaunt\corref{mycorrespondingauthor}}
\author[a3]{Fr\'ed\'eric Ouimet}

\address[a2]{The University of Manchester, Manchester, M13 9PL, UK}
\address[a3]{Universit\'e du Qu\'ebec \`a Trois-Rivi\`eres, Trois-Rivi\`eres, QC G8Z 4M3, Canada\vspace{-5mm}}

\cortext[mycorrespondingauthor]{Corresponding author. Email address: robert.gaunt@manchester.ac.uk}

\begin{abstract}
We prove monotonicity properties for medians of gamma sums and differences of the form $Z_{\alpha} = \alpha X_1 + (2 - \alpha)X_2$, where $X_1$ and $X_2$ are independent gamma random variables with common shape parameter. By combining these monotonicity properties with known bounds for the median of the gamma distribution, we establish sharp bounds for the median of the variance-gamma and McKay Type I distributions. Also, by exploiting the normal variance-mean mixture representation of the generalized hyperbolic distribution together with bounds for a ratio of modified Bessel functions of the second kind, we obtain sharp bounds for the median of the generalized hyperbolic distribution. We thus resolve all five conjectures of Gaunt and Merkle \cite{MR4141494}. As a by-product of our analysis, we show that the variance-gamma distributions with positive asymmetry parameter and the McKay Type I distributions satisfy the ``mode-median-mean'' inequality for all admissible parameter values, and that the same is true of the generalized hyperbolic distribution with positive asymmetry parameter.
\end{abstract}


\end{frontmatter}

\section{Introduction}\label{sec:intro}

\subsection{The generalized hyperbolic, variance-gamma and McKay Type I distributions}

Introduced in 1977 by Barndorff-Nielsen \cite{b77}, the generalized hyperbolic (GH) distribution is widely used across the mathematical sciences in areas such as modelling of the distribution of particle size from aeolian sand deposits \cite{b77,bb85}, modelling of turbulence \cite{bjs89} and is of particular interest in financial modelling \cite{bs03,ek95,ep02,mfe05}. For parameters $\lambda\in\R$, $\alpha \geq 0$, $\beta\in\R$, $|\beta| \leq \alpha$, $\delta \geq 0$ and $\mu\in\R$ satisfying the admissibility conditions below, the GH distribution, denoted by $\mathrm{GH}(\lambda,\alpha,\beta,\delta,\mu)$, has density
\begin{equation}\label{ghpdf}
p_{\mathrm{GH}}(x) = \frac{(\gamma/\delta)^{\lambda}}{\sqrt{2\pi}K_{\lambda}(\delta\gamma)}e^{\beta(x-\mu)}\frac{K_{\lambda-1/2}(\alpha\sqrt{\delta^2 + (x-\mu)^2})}{(\sqrt{\delta^2 + (x-\mu)^2}/\alpha)^{1/2-\lambda}}, \qquad x\in\R,
\end{equation}
where $\gamma = \sqrt{\alpha^2-\beta^2}$ and $K_{\lambda}$ is the modified Bessel function of the second kind (see \cite[Chap.~10]{NISTDLMF}). As noted by \cite{eh04}, the three permitted parameter regimes are
\[
\mathrm{(i)} ~\delta \geq 0, \,\gamma > 0, \,\lambda > 0, \qquad \mathrm{(ii)} ~\delta > 0, \,\gamma > 0, \,\lambda = 0, \qquad \mathrm{(iii)} ~\delta > 0, \,\gamma \geq 0, \,\lambda < 0.
\]
If $\delta = 0$ or $\gamma = 0$, then the density \eqref{ghpdf} is defined as the limit obtained by using the limiting form $K_{\lambda}(x)\sim 2^{|\lambda|-1} \Gamma(|\lambda|)x^{-|\lambda|}$, $x\downarrow0$, $\lambda\not = 0$ (see \cite[Eqs.\ 10.27.3, 10.30.2]{NISTDLMF}).

The GH distribution enjoys the following normal variance-mean mixture representation \cite{bks82}. Let $N\sim \mathcal{N}(0,1)$ be independent of the generalized inverse Gaussian random variable $W\sim\GIG(\lambda,\delta,\gamma)$ with density
\[
p_{\mathrm{GIG}}(x) = \frac{(\gamma/\delta)^{\lambda}}{2K_{\lambda}(\delta\gamma)}x^{\lambda-1}\exp\left(-\frac{\delta^2/x + \gamma^2x}{2}\right), \qquad x > 0;
\]
for $\delta = 0$ and $\lambda > 0$, this notation is understood in the limiting sense and $W\sim \Gamma(\lambda,\gamma^2/2)$ (with the parameterization of the gamma distribution as in \eqref{gampdf} below). For $\gamma = 0$ and $\lambda < 0$, this notation is also understood in the limiting sense, and $W$ has density
\[
\frac{(\delta^2/2)^{-\lambda}}{\Gamma(-\lambda)} x^{\lambda-1} \exp\left(-\frac{\delta^2}{2x}\right), \qquad x > 0.
\]
Then
\begin{equation}\label{eq:definitions.GH.mixture}
\mu + \beta W + \sqrt{W} N \sim \mathrm{GH}(\lambda,\alpha,\beta,\delta,\mu).
\end{equation}
The normal variance-mean mixture representation \eqref{eq:definitions.GH.mixture} can be useful when deriving distributional properties, as conditional on $W$ the GH random variable $\mu + \beta W + \sqrt{W} N$ is normally distributed; we will exploit this property of the GH distribution in this paper.

The GH distribution with $\delta = 0$ corresponds to the variance-gamma (VG) distribution; further special and limiting cases of the GH distribution include the hyperbolic, normal-inverse Gaussian and Student's $t$-distributions \cite{bs03,eh04}. In this paper, it will be convenient to work with the following parameterization of the VG distribution in which $r = 2\lambda$, $\theta = \beta/\gamma^2$ and $\sigma = 1/\gamma$; for this and other parameterizations see \cite{FischerGauntSarantsev2025VGReview}. Under this parameterization, the VG distribution, which we denote by $\mathrm{VG}(r,\theta,\sigma,\mu)$, has density
\[
p_{\mathrm{VG}}(x) = \frac{1}{\sigma\sqrt{\pi} \Gamma(r/2)} e^{\theta (x-\mu)/\sigma^2} \bigg(\frac{|x-\mu|}{2\sqrt{\theta^2 + \sigma^2}}\bigg)^{(r-1)/2} K_{\frac{r-1}{2}}\bigg(\frac{\sqrt{\theta^2 + \sigma^2}}{\sigma^2} |x-\mu| \bigg), \qquad x\in\R.
\]
Like the GH distribution, the VG distribution has been widely used in financial modelling \cite{mcc98,ms90} and also possesses an attractive distributional theory, with several important distributions arising as special and limiting cases, including the normal, gamma and asymmetric Laplace distributions and the product of two zero mean normal random variables \cite{FischerGauntSarantsev2025VGReview}. The VG distribution is also referred to as the generalized Laplace distribution \cite{kkp01}, the Bessel function distribution \cite{m32} and the McKay Type II distribution \cite{HolmAlouini2004}. In this paper, we shall let $\mu = 0$ throughout, as results in the general case follow from the basic relations $\mathrm{GH}(\lambda,\alpha,\beta,\delta,\mu) \, \smash{\stackrel{d}{=}} \, \mu + \mathrm{GH}(\lambda,\alpha,\beta,\delta,0)$ and $\mathrm{VG}(r,\theta,\sigma,\mu) \, \smash{\stackrel{d}{=}} \, \mu + \mathrm{VG}(r,\theta,\sigma,0)$. We shall also suppose that $\beta > 0$ and $\theta > 0$, since $\mathrm{GH}(\lambda,\alpha,\beta,\delta,0) \, \smash{\stackrel{d}{=}} \, -\mathrm{GH}(\lambda,\alpha,-\beta,\delta,0)$ and $\mathrm{VG}(r,\theta,\sigma,0) \, \smash{\stackrel{d}{=}} \, -\mathrm{VG}(r,-\theta,\sigma,0)$.

The McKay Type I distribution \cite{m32} is defined similarly to the VG distribution but with the role of the modified Bessel function of the second kind $K_{\lambda}$ replaced by the modified Bessel function of the first kind $I_{\lambda}$ and the support of the distribution restricted to the positive real line.
For $b > 0$, $c > 1$, $m > -1/2$, the McKay Type I distribution has density
\begin{equation}\label{mden}
p_{\mathrm{McKay} \, \mathrm{I}}(x) = \frac{\sqrt{\pi}(c^2-1)^{m + 1/2}}{2^m b^{m + 1} \Gamma(m + 1/2)}x^m e^{-cx/b}I_m\bigg(\frac{x}{b}\bigg), \qquad x > 0.
\end{equation}
Distributional theory for the McKay Type I distribution was developed by \cite{HolmAlouini2004}, motivated by an application to wireless communication systems, and further recent contributions to the distributional theory of the McKay Type I distribution include those of \cite{b26,pog1,pog2}.

\subsection{Conjectured bounds for the median}

The presence of the modified Bessel functions in the densities of the GH, VG and McKay Type I distributions means that explicit formulas for basic distributional summaries can be difficult to obtain. Whilst closed-form expressions for the moments of the GH, VG and McKay Type I distributions are known (see \cite{bs03}, \cite{FischerGauntSarantsev2025VGReview} and \cite{HolmAlouini2004}, respectively), exact formulas for the mode and median are not available for general parameter constellations, with notable exceptions being that the mode and median of the GH and VG distributions are clearly given by the location parameter $\mu$ when the asymmetry parameters $\beta$ and $\theta$ are zero. Since one cannot hope for closed-form formulas for the mode and median of the GH, VG and McKay Type I distributions for general parameter constellations, \cite{MR4141494} sought to derive tight two-sided bounds for the mode and median of these distributions, and succeeded in achieving such bounds for the mode. However, attaining bounds for the median proved to be more challenging, so the authors contented themselves with conjecturing tight bounds for the medians of these distributions.

The conjectured bounds for the median of the VG and McKay Type I distributions were based on representations of these distributions in terms of differences and sums of independent gamma random variables, which we now recall. For $r > 0$ and $\lambda > 0$, we write $G\sim \Gamma(r,\lambda)$ if $G$ has density
\begin{equation}\label{gampdf}
p_G(x) = \frac{\lambda^r}{ \Gamma(r)}x^{r-1}e^{-\lambda x}, \qquad x > 0.
\end{equation}
As noted by \cite[Eq.\ (3.43)]{MR4141494}, it is almost immediate from \cite[Proposition 1.2, part (vi)]{g14} that a VG distributed random variable $V_{r,\theta,\sigma}\sim\mathrm{VG}(r,\theta,\sigma,0)$ satisfies the distributional relation
\begin{equation}\label{eq:definitions.VG.gamma.representation}
V_{r,\theta,\sigma} \stackrel{d}{=} \theta\left[\left(\sqrt{1 + \kappa} + 1\right)Y_1 - \left(\sqrt{1 + \kappa} - 1\right)Y_2\right], \qquad \kappa \leqdef \frac{\sigma^2}{\theta^2},
\end{equation}
where $Y_1,Y_2$ are independent $\Gamma(r/2,1)$ random variables. Similarly, as observed by \cite[Eq.\ (3.44)]{MR4141494}, an application of \cite[Theorem 3]{HolmAlouini2004} reveals that a McKay Type I random variable $Z_{m,c,\phi}$ with density \eqref{mden} enjoys the gamma-sum representation
\begin{equation}\label{eq:definitions.McKay.gamma.representation}
Z_{m,c,\phi} \stackrel{d}{=} \frac{\phi(c - 1)}{c}Y_1 + \frac{\phi(c + 1)}{c}Y_2, \qquad \phi \leqdef \frac{bc}{c^2-1},
\end{equation}
where $Y_1,Y_2$ are independent $\Gamma(m + 1/2,1)$ random variables. As noted by \cite{MR4141494}, from the distributional relations \eqref{eq:definitions.VG.gamma.representation}--\eqref{eq:definitions.McKay.gamma.representation} and basic properties of the gamma distribution, $V_{r,\theta,\sigma} \, \smash{\stackrel{d}{\to}} \, \Gamma(r/2,(2\theta)^{-1})$ as $\sigma\downarrow0$, whilst $Z_{m,c,\phi} \, \smash{\stackrel{d}{\to}} \, \Gamma(m + 1/2,(2\phi)^{-1})$ as $c\downarrow1$, and $Z_{m,c,\phi} \, \smash{\stackrel{d}{\to}} \, \Gamma(2m + 1,1/\phi)$ as $c\to\infty$.

The first two conjectures of \cite{MR4141494} concerned monotonicity of the median of the random variable
\begin{equation}\label{eq:definitions.Z.alpha}
Z_{\alpha} \leqdef \alpha X_1 + (2 - \alpha)X_2, \qquad \alpha\in\R,
\end{equation}
where $X_1$ and $X_2$ are independent $\Gamma(r,1)$ random variables. For a real-valued random variable $X$ with cumulative distribution function (cdf) $F_X(x) = \PP(X \leq x)$, $x\in\R$, the median of $X$ is defined by $\Med(X) \leqdef \inf\{x\in\R:F_X(x) \geq 1/2\}$. If $X$ has a continuous strictly positive density, then $\Med(X)$ is the unique real number $m$ such that $\PP(X \leq m) = 1/2$. Conjecture~3.2 of \cite{MR4141494} asserts that the function $\alpha\mapsto\Med(Z_{\alpha})$ is non-decreasing on $(0,1)$, whilst Conjecture 3.3 asserts that $\alpha\mapsto\Med(Z_{\alpha})$ is non-increasing on $(2,\infty)$. Note that, by exchangeability of $X_1$ and $X_2$, it would follow that $\alpha\mapsto\Med(Z_{\alpha})$ is non-increasing on $(1,2)$ if Conjecture 3.2 were true.

Since the limiting gamma cdfs are continuous and strictly increasing at their medians, their quantile functions are continuous at $1/2$. Therefore, the convergence of quantiles under weak convergence \cite[Lemma~21.2]{vdv98} shows that a resolution of Conjectures 3.2 and 3.3 would lead to the two-sided inequalities
\begin{align}\label{vineq}
\lim_{\sigma\to\infty}\Med(V_{r,\theta,\sigma}) < \Med(V_{r,\theta,\sigma}) < \lim_{\sigma\downarrow0}\Med(V_{r,\theta,\sigma}) = \Med(G),
\end{align}
where $G\sim \Gamma(r/2,(2\theta)^{-1})$ with $\theta > 0$, and
\begin{align}\label{mineq}
\Med(G_1) = \lim_{c\downarrow1}\Med(Z_{m,c,\phi}) < \Med(Z_{m,c,\phi}) < \lim_{c\to\infty}\Med(Z_{m,c,\phi}) = \Med(G_2),
\end{align}
where $G_1\sim \Gamma(m + 1/2,(2\phi)^{-1})$ and $G_2\sim \Gamma(2m + 1,1/\phi)$. By appealing to the extensive literature on tight bounds for the median of the gamma distribution (see, for example, \cite{alm03,a05,bp06,cr86,MR1195477,Lyon2021,Lyon2023}) it would then be possible to deduce from inequalities \eqref{vineq} and \eqref{mineq} a tight two-sided inequality for the median of the McKay Type I distribution and a tight upper bound for the median of the VG distributions; it still remains to compute the limit $\lim_{\sigma\to\infty}\Med(V_{r,\theta,\sigma})$ in order to obtain a lower bound for the median of the VG distribution. The corresponding conjectured two-sided bounds for the medians of the VG and McKay Type I distributions were given in Conjectures 3.5 and 3.6 of \cite{MR4141494}, with the former conjecture also conjecturing that $\lim_{\sigma\to\infty}\Med(V_{r,\theta,\sigma}) = 0\vee (r-1)\theta$.

For general parameter values, the GH distribution does not enjoy a representation of the form \eqref{eq:definitions.Z.alpha}. However, by considering the conjectured lower bound for the median of the VG distribution together with the functional form of their bounds for the mode of the GH distribution, \cite{MR4141494} conjectured that, for $X\sim \mathrm{GH}(\lambda,\alpha,\beta,\delta,0)$ with $\lambda > 1/2$ and $\beta > 0$,
\[
\Med(X) > \frac{\beta}{\gamma^2}\Big[\lambda-1/2 + \sqrt{(\lambda-1/2)^2 + \delta^2\gamma^2}\Big].
\]

\subsection{Summary of our contributions and outline of the paper}

The principal contribution of this paper is the complete resolution of all five conjectures posed in \cite{MR4141494}. In Theorems \ref{prop:conj3.2} and \ref{prop:conj3.3}, we prove Conjectures 3.2 and 3.3 of \cite{MR4141494}, and we in fact prove the stronger results that the function $\alpha\mapsto\Med(Z_{\alpha})$ is strictly increasing on $(0,1)$ (and thus strictly decreasing on $(1,2)$), and strictly decreasing on $(2,\infty)$.
In proving Theorem \ref{prop:conj3.2}, we follow the proof strategy suggested by \cite[pp.\ 15]{MR4141494}, in which a connection was made between Schur convexity and their Conjecture 3.2 that reduced proving the conjecture to proving that $\Med(Z_{\alpha})\in[0,2r]$.

We prove that $\Med(Z_{\alpha})\in[0,2r]$ by making use of an alternative representation of $Z_{\alpha}$ in terms of independent gamma and beta random variables, via the so-called ``beta-gamma algebra.'' With $S \leqdef X_1 + X_2$ and $P \leqdef X_1/S$, the random variables $S$ and $P$ are independent, $S\sim \Gamma(2r,1)$ and $P\sim\operatorname{Beta}(r,r)$ (see \cite[Eq.\ (4)]{d98}). In particular, $V \leqdef 1 - 2P$ is symmetric about the origin with density $h_r(v) = 2^{1-2r}(1 - v^2)^{r-1}/B(r,r)$, $-1 < v < 1$, where $B(\cdot,\cdot)$ denotes the beta function. Thus, for $0 < \alpha < 1$,
\begin{equation}\label{eq:definitions.positive.gamma.representation}
Z_{\alpha} \stackrel{d}{=} S(1 + (1 - \alpha)V).
\end{equation}
With the representation \eqref{eq:definitions.positive.gamma.representation} at hand, probabilities involving $Z_{\alpha}$ can be calculated by conditioning on $V$, and we use this approach to convert the problem of proving the inequality $\PP(Z_{\alpha} \leq 2r) > 1/2$, $0 < \alpha < 1$, into proving a certain inequality involving the cdf of the gamma distribution, which is given by Lemma \ref{lem:gamma.reflection.bound}. We are able to prove that $\alpha\mapsto\Med(Z_{\alpha})$ is strictly increasing on $(0,1)$ (not just non-decreasing on $(0,1)$) by upgrading the Schur-convexity Lemma 3.4 of \cite{MR4141494} to include strict Schur-convexity in Lemma \ref{lem:Schur.gamma}.

We use a different approach to proving Theorem \ref{prop:conj3.3} that does not involve appealing to results on Schur-convexity, although central to our proof is an alternative representation of $Z_{\alpha}$ as a normal variance-mean mixture with gamma mixing distribution. With this representation, we can express the cdf of $Z_{\alpha}$ as an integral involving the standard normal cdf and the gamma density, and via a direct but delicate analysis of the cdf of $Z_{\alpha}$ we are able to establish that $\alpha\mapsto\Med(Z_{\alpha})$ is strictly decreasing on $(2,\infty)$.

With Theorems \ref{prop:conj3.2} and \ref{prop:conj3.3} proved, sharp bounds for the median of the McKay Type I distribution are readily deduced via \eqref{mineq} and known tight bounds for the median of the gamma distribution (see Corollary \ref{cor:conj3.6}), whilst a sharp upper bound for the median of the VG distribution follows similarly from the upper bound in \eqref{vineq} (see Corollary \ref{cor:conj3.5}). The most involved part of the proof of Corollary \ref{cor:conj3.5} is the calculation of the limit $\lim_{\sigma\to\infty}\Med(V_{r,\theta,\sigma})$, for which we again employ the normal variance-mean mixture representation of the VG distribution and a delicate analysis of the cdf of the VG distribution.

In Theorem \ref{prop:conj3.7}, we derive the lower bound for the GH median presented in \cite[Conjecture 3.7]{MR4141494}, and we further extend the range of validity of the bound beyond that stated in the conjecture. The key to the proof is Lemma \ref{lem:GH.comparison}, in which we make use of the normal variance-mean mixture representation \eqref{eq:definitions.GH.mixture} of the GH distribution. By a delicate analysis of the cdf of the GH distribution, via its normal variance-mean mixture representation, and by appealing to known bounds for a ratio of modified Bessel functions of the second kind we are able to establish sharp conditions under which certain probabilities involving a GH distributed random variable are increasing or decreasing. With Lemma \ref{lem:GH.comparison} proven, we readily deduce the desired lower bound for the GH median. The same approach can be used to derive upper bounds for the median of the GH distribution, and we present a family of such bounds in Theorem \ref{prop:GH.upper.bounds} and prove that the constants in these bounds are best possible.

Our bounds for the median of the VG, McKay Type I and GH distributions complement the bounds for the mode of these distributions that were derived by \cite{MR4141494}, and when combined with the exact formulas for the means of these distributions, one can deduce mode-median-mean inequalities for these distributions, which we present in Corollary \ref{mmm}. Whilst general methods exist for proving that the mode, median and mean of positively skewed unimodal distributions satisfy this ordering (see, for example, \cite{gm77,v78}), we were not able to apply standard sufficient conditions in our setting, and so our median bounds proved essential in allowing us to attain the mode-median-mean inequalities of Corollary \ref{mmm}.

The rest of the paper is organized as follows. In Section~\ref{sec:results}, we state our main results, the monotonicity results for medians of gamma sums and gamma differences, their VG and McKay Type I consequences, and the GH median bounds. In this section, we also present our mode-median-mean inequalities for the VG, McKay Type I and GH distributions. In Section \ref{subsec:proof.auxiliary}, we state and prove three auxiliary lemmas, which we apply in Section~\ref{sec:proofs} to prove our main results.

\section{Main results}\label{sec:results}

We begin by resolving the gamma-sum and gamma-difference monotonicity conjectures of \cite[Conjectures~3.2 and 3.3]{MR4141494}. We state the monotonicity result for the gamma-sum given in Theorem \ref{prop:conj3.2} on the full interval $(0,2)$, using the symmetry $Z_{2 - \alpha} \, \smash{\stackrel{d}{=}} \, Z_{\alpha}$ to include the mirror range $1 < \alpha < 2$.

\begin{theorem}[Conjecture~3.2 of \cite{MR4141494}]\label{prop:conj3.2}
Let $X_1$ and $X_2$ be independent $\Gamma(r,1)$ random variables, where $r > 0$, and let $Z_{\alpha}$ be defined by \eqref{eq:definitions.Z.alpha}. Then the function
\[
\alpha\mapsto \Med(Z_{\alpha})
\]
is strictly increasing on $(0,1)$, strictly decreasing on $(1,2)$, and satisfies $\Med(Z_{\alpha}) = \Med(Z_{2 - \alpha})$ for $0 < \alpha < 2$.
\end{theorem}

\begin{theorem}[Conjecture~3.3 of \cite{MR4141494}]\label{prop:conj3.3}
Let $X_1$ and $X_2$ be independent $\Gamma(r,1)$ random variables, where $r > 0$, and let $Z_{\alpha}$ be defined by \eqref{eq:definitions.Z.alpha}. Then the function
\[
\alpha\mapsto \Med(Z_{\alpha})
\]
is strictly decreasing on $(2,\infty)$.
\end{theorem}

With the monotonicity results of Theorems \ref{prop:conj3.2} and \ref{prop:conj3.3} established, we are able to deduce bounds for the median of the VG and McKay Type I distributions via inequalities \eqref{vineq} and \eqref{mineq}, known bounds for the median of the gamma function and calculation of $\lim_{\sigma\to\infty}\Med(V_{r,\theta,\sigma})$. The calculation of $\lim_{\sigma\to\infty}\Med(V_{r,\theta,\sigma})$ is non-trivial and is the most involved part of the proof of the following Corollary \ref{cor:conj3.5}.

From the extensive literature on bounds for the median of the gamma distribution, we apply the following bounds.
Let $G\sim \Gamma(r,\lambda)$. Then, it was shown by \cite{bp06} that, for $r > 0$ and $\lambda > 0$,
\begin{equation}\label{twoside1} \frac{r-\log(2)}{\lambda} < \frac{r}{\lambda}e^{-\log (2)/r} < \Med(G) < \frac{r}{\lambda}e^{-1/(3r)} < \frac{1}{\lambda}\bigg(r-\frac{1}{3} + \frac{1}{18r}\bigg),
\end{equation}
whilst it was proven by \cite{MR4141494} that, for $r \geq 1$ and $\lambda > 0$,
\begin{equation}\label{twoside2}
\frac{r-1/3}{\lambda} < \Med(G) \leq \frac{r-1 + \log(2)}{\lambda},
\end{equation}
where we have equality in the upper bound in \eqref{twoside2} if and only if $r = 1$. Previously, the two-sided inequality \eqref{twoside2} had been established by \cite{MR1195477} for $r\in\Z^{+}$. Alternative tight bounds for the median of the gamma distribution are available in the literature (see, for example, \cite{Lyon2021,Lyon2023} and references therein), and applications of these bounds would lead to alternative bounds for the median of the VG and McKay Type I distributions that could improve on the bounds of Corollaries \ref{cor:conj3.5} and \ref{cor:conj3.6} in certain parameter regimes.

\begin{corollary}[Conjecture~3.5 of \cite{MR4141494}]\label{cor:conj3.5}
Let $V_{r,\theta,\sigma}\sim \mathrm{VG}(r,\theta,\sigma,0)$, where $r > 0$, $\theta > 0$ and $\sigma > 0$. Then the function
\begin{equation}\label{eq:cor:conj3.5.claim.1}
\sigma\mapsto \Med(V_{r,\theta,\sigma})
\end{equation}
is strictly decreasing on $(0,\infty)$. Consequently, if $G\sim \Gamma(r/2,(2\theta)^{-1})$, then
\begin{equation}\label{eq:cor:conj3.5.claim.2}
\Med(V_{r,\theta,\sigma}) < \Med(G), \qquad \sigma > 0.
\end{equation}
Moreover,
\begin{equation}\label{eq:asserted.Med}
\lim_{\sigma\to\infty}\Med(V_{r,\theta,\sigma}) =
\begin{cases}
(r - 1)\theta, & r > 1,\\
0, & 0 < r \leq 1.
\end{cases}
\end{equation}
Therefore, for $r > 0$,
\begin{equation}\label{eq:results.VG.bounds.general}
0\vee(r - 1)\theta < \Med(V_{r,\theta,\sigma}) < r\theta e^{-2/(3r)} < \left(r - \frac{2}{3} + \frac{2}{9r}\right)\theta,
\end{equation}
and, for $r \geq 2$,
\begin{equation}\label{eq:results.VG.bounds.refined}
\Med(V_{r,\theta,\sigma}) \leq (r + 2\log(2) - 2)\theta.
\end{equation}
\end{corollary}

\begin{corollary}[Conjecture~3.6 of \cite{MR4141494}]\label{cor:conj3.6}
Fix $m > -1/2$ and $\phi > 0$. For each $c > 1$, let $Z_{m,c,\phi}$ follow the McKay Type I distribution with density \eqref{mden}, where $b = \phi(c^2 - 1)/c$. Then the function
\[
c\mapsto \Med(Z_{m,c,\phi})
\]
is strictly increasing on $(1,\infty)$. Consequently, if $G_1\sim \Gamma(m + 1/2,(2\phi)^{-1})$ and $G_2\sim \Gamma(2m + 1,1/\phi)$, then
\[
\Med(G_1) < \Med(Z_{m,c,\phi}) < \Med(G_2), \qquad c > 1.
\]
Therefore, for $m > -1/2$,
\begin{equation}\label{eq:results.McKay.bounds.general}
\begin{aligned}
(2m + 1 - 2\log(2))\phi & < (2m + 1)\phi e^{-2\log(2)/(2m + 1)} \\[1mm]
& < \Med(Z_{m,c,\phi}) \\[-2mm]
& < (2m + 1)\phi e^{-1/(3(2m + 1))} < \left(2m + \frac{2}{3} + \frac{1}{18(2m + 1)}\right)\phi,
\end{aligned}
\end{equation}
and, for $m \geq 1/2$,
\begin{equation}\label{eq:results.McKay.bounds.refined}
(2m + 1/3)\phi < \Med(Z_{m,c,\phi}) \leq (2m + \log(2))\phi.
\end{equation}
\end{corollary}

Let $(X, Y)$ be a bivariate normal random vector with zero-mean vector, positive variances $(\sigma_X^2, \sigma_Y^2)$, and correlation coefficient $-1 < \rho < 1$. Since the work of \cite{craig,wb32}, the distribution of the product $Z = XY$ and, more generally, the sum of independent copies of such random variables, has received interest in the statistics literature and has been widely used in applications throughout the mathematical sciences; see \cite{g22,np16} for application areas and the former reference for basic distributional theory. Let $Z_1,\ldots,Z_n$ be independent copies of the product $Z$. Then it was shown by \cite{gaunt prod} that the sample mean $\overline{Z}_n = n^{-1}\sum_{i=1}^nZ_i$ is VG distributed, $\overline{Z}_n\sim\mathrm{VG}(n,\rho\sigma_X\sigma_Y/n,\sigma_X\sigma_Y\sqrt{1-\rho^2}/n,0)$. Thus, the following corollary is an immediate consequence of Corollary \ref{cor:conj3.5}.

\begin{corollary}Let $\overline{Z}_n$ be defined as above. Let $0 < \rho < 1$. Then, for $n \geq 1$,
\[
0 \vee \rho\sigma_X\sigma_Y\bigg(1-\frac{1}{n}\bigg) < \Med(\overline{Z}_n) < \rho\sigma_X\sigma_Y e^{-2/(3n)} < \rho\sigma_X\sigma_Y\left(1 - \frac{2}{3n} + \frac{2}{9n^2}\right),
\]
and, for $n \geq 2$,
\[
\Med(\overline{Z}_n) \leq \rho\sigma_X\sigma_Y \bigg(1 - \frac{2(1-\log(2))}{n}\bigg).
\]
\end{corollary}

Our next result proves the lower bound for the median of the GH distribution conjectured by \cite[Conjecture~3.7]{MR4141494}. The statement of the theorem further upgrades \cite[Conjecture~3.7]{MR4141494} by showing that, when $\delta > 0$, the restriction $\lambda > 1/2$ stated in the conjecture is not needed.

\begin{theorem}[Conjecture~3.7 of \cite{MR4141494}, with extension to all $\lambda\in\R$ for $\delta > 0$]\label{prop:conj3.7}
Let $X\sim \mathrm{GH}(\lambda,\alpha,\beta,\delta,0)$, where $\alpha > \beta > 0$ and $\gamma = \sqrt{\alpha^2 - \beta^2}$. Assume either that $\delta > 0$ and $\lambda\in\R$, or that $\delta = 0$ and $\lambda > 1/2$. Then
\begin{equation}\label{eq:results.conj3.7.bound}
\Med(X) > \frac{\beta}{\gamma^2}\Big[\lambda-1/2 + \sqrt{(\lambda-1/2)^2 + \delta^2\gamma^2}\Big].
\end{equation}
\end{theorem}

In the following theorem, we complement the lower bound of Theorem \ref{prop:conj3.7} with upper bounds for the median of the GH distribution. We introduce the notation $u_c$ to make explicit that $c$ is the fixed shift being compared in the upper-bound family. The first bound is valid uniformly over all $\lambda\in\R$ when $\delta > 0$. The second identifies the exact range of validity of the sharper choice $c = 0$ in the same family.

\begin{theorem}\label{prop:GH.upper.bounds}
Let $X\sim \mathrm{GH}(\lambda,\alpha,\beta,\delta,0)$, where $\lambda\in\R$, $\alpha > \beta > 0$, $\delta > 0$, and $\gamma = \sqrt{\alpha^2 - \beta^2}$. For $c\in\R$, set
\[
u_c \leqdef \frac{\beta}{\gamma^2}\left[\lambda - c + \sqrt{(\lambda - c)^2 + \delta^2\gamma^2}\right].
\]
Then
\begin{equation}\label{eq:results.GH.upper.global}
\Med(X) < u_{-1/2} = \frac{\beta}{\gamma^2}\left[\lambda + 1/2 + \sqrt{(\lambda + 1/2)^2 + \delta^2\gamma^2}\right].
\end{equation}
Moreover,
\[
u_0 = \frac{\beta}{\gamma^2}\left[\lambda + \sqrt{\lambda^2 + \delta^2\gamma^2}\right],
\]
and
\begin{equation}\label{eq:results.GH.upper.zero}
\begin{cases}
\Med(X) < u_0, & \lambda > 0,\\
\Med(X) = u_0, & \lambda = 0,\\
\Med(X) > u_0, & \lambda < 0.
\end{cases}
\end{equation}
Furthermore, $c = -1/2$ is the largest real constant for which the upper bound $\Med(X) < u_c$ is valid uniformly over all $\lambda\in\R$ and $\delta > 0$. On the range $\lambda \geq 1$, the largest real constant for which $\Med(X) < u_c$ is valid uniformly over all $\delta > 0$ is $c = 0$.
\end{theorem}

\begin{remark}\label{rem:GH.upper.constants}
The refined variance-gamma bound \eqref{eq:results.VG.bounds.refined}, after the change of parameters $r = 2\lambda$ and $\theta = \beta/\gamma^2$, gives the limiting $\delta = 0$ upper bound $\Med(X) < u_{1 - \log(2)}$ for $\lambda \geq 1$. Theorem~\ref{prop:GH.upper.bounds} shows that this positive value of $c$ cannot be promoted to a bound that is uniform over the GH family with $\delta > 0$.
\end{remark}

With all of our bounds for the median of the VG, McKay Type I and GH distributions presented, we are now able to deduce the following mode-median-mean inequalities, by combining our median bounds with those of \cite{MR4141494} for the mode, together with exact formulas for the means of these distributions.

\begin{corollary}\label{mmm} 1. Let $V_{r,\theta,\sigma}\sim \mathrm{VG}(r,\theta,\sigma,0)$, with $r > 0$, $\theta > 0$ and $\sigma > 0$. Then
\begin{equation}\label{mmm1}
\mathrm{Mode}(V_{r,\theta,\sigma}) < \Med(V_{r,\theta,\sigma}) < \EE[V_{r,\theta,\sigma}].
\end{equation}
2. Let $Z_{m,c,b}$ follow the McKay Type I distribution with density \eqref{mden} with parameters $m > -1/2$, $c > 1$ and $b > 0$. Then
\begin{equation}\label{mmm2}
\mathrm{Mode}(Z_{m,c,b}) < \Med(Z_{m,c,b}) < \EE[Z_{m,c,b}].
\end{equation}
3. Let $X\sim \mathrm{GH}(\lambda,\alpha,\beta,\delta,0)$, where $\lambda\in\R$, $\alpha > \beta > 0$, $\delta > 0$. Then
\begin{equation}\label{mmm3}
\mathrm{Mode}(X) < \Med(X).
\end{equation}
Moreover,
\begin{equation}\label{mmm4}
\Med(X) < \EE[X],
\end{equation}
so that
\begin{equation}
\mathrm{Mode}(X) < \Med(X) < \EE[X]. \label{mmm5}
\end{equation}
\end{corollary}

\section{Auxiliary results}\label{subsec:proof.auxiliary}

We begin this section by stating a lemma concerning Schur-convexity of gamma sums. The lemma upgrades Lemma 3.4 of \cite{MR4141494} to include a strict Schur-convexity assertion, which allows us to prove Theorem \ref{prop:conj3.2} with a strictly increasing assertion rather than a merely non-decreasing statement as was given in \cite[Conjecture 3.2]{MR4141494}. In order to state the lemma, we recall the relevant two-dimensional form of the majorization order. For $\bb{x} = (x_1,x_2)$ and $\bb{y} = (y_1,y_2)$ in $\R^2$, we say that $\bb{x}$ is majorized by $\bb{y}$, and write $\bb{x}\prec \bb{y}$, if
\[
x_1 + x_2 = y_1 + y_2, \qquad \max\{x_1,x_2\} \leq \max\{y_1,y_2\}.
\]
Equivalently, after arranging both vectors in decreasing order, the largest coordinate of $\bb{x}$ is no larger than the largest coordinate of $\bb{y}$, and the two coordinate sums are equal. A function $f:D\to\R$, where $D\subseteq\R^2$, is called Schur-convex on $D$ if, for all $\bb{x},\bb{y}\in D$ such that $\bb{x}\prec \bb{y}$, one has $f(\bb{x}) \leq f(\bb{y})$. It is strictly Schur-convex on a fixed-sum line if this inequality is strict whenever $\bb{x}\prec \bb{y}$ and $\bb{x}$ is not a permutation of $\bb{y}$. For further details on Schur convexity and its applications, we refer the reader to the classic monograph \cite{mo09}.

\begin{lemma}[Schur-convexity of gamma sums]\label{lem:Schur.gamma}
Let $X_1$ and $X_2$ be independent $\Gamma(r,1)$ random variables. For $\bb{c} = (c_1,c_2)$ with $c_1,c_2 > 0$, define
\[
F(\bb{c};t) \leqdef \PP(c_1X_1 + c_2X_2 \leq t).
\]
Then, for every $s > 0$ and every $0 \leq t \leq rs$, the restriction of $\bb{c}\mapsto F(\bb{c};t)$ to the fixed-sum line $\{\bb{c}\in(0,\infty)^2:c_1 + c_2 = s\}$ is Schur-convex. More explicitly, if $\bb{c},\bb{d}\in(0,\infty)^2$ satisfy $\bb{c}\prec \bb{d}$, $c_1 + c_2 = d_1 + d_2 = s$, and $0 \leq t \leq rs$, then $F(\bb{c};t) \leq F(\bb{d};t)$. Moreover, if $0 < t < rs$ and $\bb{c}$ is not a permutation of $\bb{d}$, then the comparison is strict: $F(\bb{c};t) < F(\bb{d};t)$.
\end{lemma}

\begin{proof}
The Schur-convexity assertion is quoted in this form as \cite[Lemma~3.4]{MR4141494}, taken from \cite{BockDiaconisHufferPerlman1987}. To establish the strict comparison, we use \cite[Theorem~1, parts (a) and (b)]{BockDiaconisHufferPerlman1987}. Since $X_1$ and $X_2$ are independent and identically distributed, $F(\bb{c};t)$ is invariant under permutations of the coordinates of $\bb{c}$. We may therefore suppose, without loss of generality, that $c_1 \geq c_2$ and $d_1 \geq d_2$. Set $p \leqdef c_1/s$ and $q \leqdef d_1/s$. Then $1/2 \leq p \leq q < 1$, where $p \leq q$ follows from $\bb{c} \prec \bb{d}$. Moreover, $p < q$ if and only if $\bb{c}$ is not a permutation of $\bb{d}$.

For $v \in [1/2,1)$, define $H_v(u) \leqdef \PP\bigl(vX_1 + (1 - v)X_2 \leq u\bigr)$. By \cite[Theorem~1, parts (a) and (b)]{BockDiaconisHufferPerlman1987}, for each $v \in (1/2,1)$ there exists a value $t_0(v) > r$ such that $\frac{\partial}{\partial v}H_v(u) > 0$ for $0 < u < t_0(v)$. Consequently, for each fixed $0 < u \leq r$, the function $v\mapsto H_v(u)$ is strictly increasing on $(1/2,1)$. Since $v\mapsto H_v(u)$ is continuous on $[1/2,1)$, it is strictly increasing on $[1/2,1)$. Also, $H_v(0) = 0$ for all $v \in [1/2,1)$.

By exchangeability of $X_1$ and $X_2$ and a rescaling, we have $F(\bb{c};t) = H_p(t/s)$ and $F(\bb{d};t) = H_q(t/s)$. It follows that $F(\bb{c};t) \leq F(\bb{d};t)$ for $0 \leq t \leq rs$. If $0 < t < rs$ and $\bb{c}$ is not a permutation of $\bb{d}$, then $p < q$, and therefore $F(\bb{c};t) < F(\bb{d};t)$.

We use this comparison in the coefficient-sum case $c_1 + c_2 = 2$. If $0 < \alpha_2 < \alpha_1 < 1$, define $\bb{c}_{\alpha_i} \leqdef (\alpha_i,2 - \alpha_i)$, $i \in \{1,2\}$. Then $\bb{c}_{\alpha_1} \prec \bb{c}_{\alpha_2}$, because the two coordinate sums are equal to $2$ and $\max\{\alpha_1,2 - \alpha_1\} = 2 - \alpha_1 < 2 - \alpha_2 = \max\{\alpha_2,2 - \alpha_2\}$. The two coefficient vectors $\bb{c}_{\alpha_1}$ and $\bb{c}_{\alpha_2}$ are not permutations of one another. Hence the strict Schur-convexity comparison gives
\begin{equation}\label{eq:Schur.special.case}
\PP(Z_{\alpha_1} \leq t) = F(\bb{c}_{\alpha_1};t) < F(\bb{c}_{\alpha_2};t) = \PP(Z_{\alpha_2} \leq t), \qquad 0 < t < 2r.
\end{equation}
This concludes the proof.
\end{proof}

The gamma distribution reflection estimate given in the following lemma will be used in conjunction with Lemma \ref{lem:Schur.gamma} to prove Theorem \ref{prop:conj3.2}.

\begin{lemma}\label{lem:gamma.reflection.bound}
Let $G_k$ denote the cdf of the $\Gamma(k,1)$ distribution. For every $k > 0$ and $u\in[0,1)$,
\begin{equation}\label{eq:gamma.reflection.bound}
G_k\left(\frac{k}{1 + u}\right) + G_k\left(\frac{k}{1 - u}\right) > 1.
\end{equation}
\end{lemma}

\begin{proof}
In this proof, we will let $g_k$ denote the density of the $\Gamma(k,1)$ distribution. We also let
\[
H(u) \leqdef G_k\left(\frac{k}{1 + u}\right) + G_k\left(\frac{k}{1 - u}\right), \qquad 0 \leq u < 1.
\]
We first show that $G_k(k) > 1/2$. By the change of variables $y = k^2/x$,
\[
\int_k^{\infty} y^{k-1}e^{-y} \, \rd y = k^{2k}\int_0^k x^{-k-1}e^{-k^2/x} \, \rd x.
\]
For $0 < x < k$, with $t = x/k$,
\[
\log\left(\frac{x^{k-1}e^{-x}}{k^{2k}x^{-k-1}e^{-k^2/x}}\right) = k\left(\frac{1}{t} - t + 2\log(t)\right) > 0,
\]
because the function $\psi(t) \leqdef t^{-1} - t + 2\log(t)$ satisfies $\psi(1) = 0$ and $\psi'(t) = -(1 - t)^2/t^2 \leq 0$ on $(0,1)$. Hence
\[
\int_0^k x^{k-1}e^{-x} \, \rd x > \int_k^{\infty} y^{k-1}e^{-y} \, \rd y,
\]
and so $G_k(k) > 1/2$. Therefore $H(0) = 2G_k(k) > 1$. Also $\lim_{u\uparrow1}H(u) = G_k(k/2) + 1 > 1$. It remains to check that $H$ has no interior minimum. Differentiation gives
\[
H'(u) = \frac{k}{(1 - u)^2}g_k\left(\frac{k}{1 - u}\right) - \frac{k}{(1 + u)^2}g_k\left(\frac{k}{1 + u}\right).
\]
Let
\[
A(u) \leqdef \frac{k}{(1 - u)^2}g_k\left(\frac{k}{1 - u}\right), \qquad B(u) \leqdef \frac{k}{(1 + u)^2}g_k\left(\frac{k}{1 + u}\right).
\]
Then $H'(u) = A(u) - B(u)$, and $A(u),B(u) > 0$. Hence the sign of $H'(u)$ is the sign of $A(u)/B(u) - 1$, equivalently the sign of $\log(A(u)/B(u))$. Using that $g_k(x) = x^{k-1}e^{-x}/ \Gamma(k)$, $x > 0$, we obtain
\[
\begin{aligned}
\log\left(\frac{A(u)}{B(u)}\right)
&= 2\log\left(\frac{1 + u}{1 - u}\right) + (k - 1)\log\left(\frac{1 + u}{1 - u}\right) - \frac{k}{1 - u} + \frac{k}{1 + u} \\
&= (k + 1)\log\left(\frac{1 + u}{1 - u}\right) - \frac{2ku}{1 - u^2}.
\end{aligned}
\]
Therefore the sign of $H'(u)$ is the sign of
\[
D(u) \leqdef (k + 1)\log\left(\frac{1 + u}{1 - u}\right) - \frac{2ku}{1 - u^2}.
\]
Moreover,
\[
D'(u) = \frac{2(1 - (2k + 1)u^2)}{(1 - u^2)^2}.
\]
Thus $D$ is strictly increasing on $(0,(2k + 1)^{-1/2})$ and strictly decreasing on $((2k + 1)^{-1/2},1)$. Since $D(0) = 0$ and $D(u)\to -\infty$ as $u\uparrow1$, the derivative $H'$ changes sign at most once, and if it changes sign then it changes from positive to negative. Hence $H$ has no interior minimum. The minimum of $H$ on $[0,1)$ is therefore attained at an endpoint in the limiting sense, where we have already shown that the value is strictly larger than $1$. This proves \eqref{eq:gamma.reflection.bound}.
\end{proof}

In anticipation of the forthcoming Lemma \ref{lem:GH.comparison}, we recall some inequalities for a ratio of modified Bessel functions of the second kind. For $x > 0$ and $\nu\in\R$,
\begin{equation}\label{eq:Bessel.ratio.two.sided}
\frac{\nu + \sqrt{\nu^2 + x^2}}{x} < \frac{K_{\nu + 1}(x)}{K_{\nu}(x)} < \frac{\nu + 1 + \sqrt{(\nu + 1)^2 + x^2}}{x},
\end{equation}
where the lower bound is due to \cite{rs16} and the upper bound is given in \cite{ln10}.
Moreover, the following inequality can be found in \cite{s11}: for $x > 0$,
\begin{equation}\label{eq:Bessel.ratio.half}
\frac{K_{\nu + 1}(x)}{K_{\nu}(x)}
\begin{cases}
< \dfrac{\nu + 1/2 + \sqrt{(\nu + 1/2)^2 + x^2}}{x}, & \nu > -1/2, \\[4mm]
= \dfrac{\nu + 1/2 + \sqrt{(\nu + 1/2)^2 + x^2}}{x}, & \nu = -1/2, \\[4mm]
> \dfrac{\nu + 1/2 + \sqrt{(\nu + 1/2)^2 + x^2}}{x}, & \nu < -1/2.
\end{cases}
\end{equation}
For further bounds for the ratio $K_{\nu + 1}(x)/K_{\nu}(x)$ and application areas, we refer the reader to \cite{s23} and references therein.

We are now ready to state our final auxiliary lemma, which will be used in the proofs of Theorems \ref{prop:conj3.7} and \ref{prop:GH.upper.bounds}.

\begin{lemma}\label{lem:GH.comparison}
Let $\delta > 0$, $\lambda\in\R$, $\gamma > 0$ and $a\in\R$. Set $\nu \leqdef \lambda - 1/2$, $D \leqdef \delta\gamma$ and
\[
\xi_a \leqdef \frac{a + \sqrt{a^2 + D^2}}{\gamma^2}.
\]
Let $W\sim\GIG(\lambda,\delta,\gamma)$, let $N\sim \mathcal{N}(0,1)$ be independent of $W$, and define, for $b \geq 0$,
\[
X_b \leqdef bW + \sqrt{W}N, \qquad \Psi_a(b) \leqdef \PP(X_b \leq b\xi_a) = \EE\left[\Phi\left(b\frac{\xi_a - W}{\sqrt{W}}\right)\right],
\]
where $\Phi$ is the standard normal cdf. If
\begin{equation}\label{eq:GH.comparison.lower.condition}
\frac{K_{\nu + 1}(x)}{K_{\nu}(x)} > \frac{a + \sqrt{a^2 + x^2}}{x}, \qquad \text{for all $x > 0$},
\end{equation}
then $b\mapsto \Psi_a(b)$ is strictly decreasing on $(0,\infty)$. If the inequality in \eqref{eq:GH.comparison.lower.condition} is reversed for all $x > 0$, then $b\mapsto \Psi_a(b)$ is strictly increasing on $(0,\infty)$. If equality holds in \eqref{eq:GH.comparison.lower.condition} for all $x > 0$, then $\Psi_a$ is constant.
\end{lemma}

\begin{proof}
Let $C$ denote a positive normalising constant which does not depend on $w$. Differentiating under the integral sign gives the following formula for $b > 0$, with the same formula holding for the right derivative at $b = 0$:
\begin{align}
\Psi_a'(b)
&= C\int_0^{\infty} \frac{\xi_a - w}{\sqrt{w}}\exp\left(-\frac{b^2(\xi_a - w)^2}{2w}\right)w^{\lambda-1}\exp\left(-\frac{\delta^2/w + \gamma^2w}{2}\right) \, \rd w \nonumber \\
&= Ce^{b^2\xi_a}\int_0^{\infty} (\xi_a - w)w^{\nu-1}\exp\left(-\frac{(\gamma^2 + b^2)w + (\delta^2 + b^2\xi_a^2)/w}{2}\right) \, \rd w. \label{inttt}
\end{align}
Differentiation under the integral sign is justified by dominated convergence, locally uniformly for $b \geq 0$, where at $b = 0$ the derivative is understood as a right derivative. Indeed, on every compact subinterval of $[0,\infty)$, the absolute value of the differentiated integrand is bounded by a constant multiple of $(1 + w)w^{\lambda-3/2}\exp(-\delta^2/(4w)- \gamma^2w/4)$, which is integrable on $(0,\infty)$ because $\delta > 0$ and $\gamma > 0$.

Define
\[
A_b \leqdef \gamma^2 + b^2, \qquad B_b \leqdef \delta^2 + b^2\xi_a^2, \qquad x_b \leqdef \sqrt{A_bB_b}.
\]
Since $b \geq 0$, we have $A_b > 0$ and $B_b > 0$. Applying the standard integral identity
\[
\int_0^{\infty} w^{\eta-1}\exp\left(-\frac{Aw + B/w}{2}\right) \, \rd w = 2\left(\frac{B}{A}\right)^{\eta/2}K_{\eta}(\sqrt{AB}), \qquad A,B > 0, \quad \eta\in\R
\]
(see \cite[Eq.\ 10.32.9]{NISTDLMF}) to \eqref{inttt} with $\eta = \nu$ and $\eta = \nu + 1$ yields
\begin{equation}\label{eq:GH.comparison.Psi.prime}
\Psi_a'(b) = 2Ce^{b^2\xi_a}\left(\frac{B_b}{A_b}\right)^{\nu/2}K_{\nu}(x_b)\left[\xi_a - \sqrt{\frac{B_b}{A_b}}\frac{K_{\nu + 1}(x_b)}{K_{\nu}(x_b)}\right].
\end{equation}
The prefactor in \eqref{eq:GH.comparison.Psi.prime} is strictly positive, so it remains to determine the sign of the expression in square brackets. By the definition of $\xi_a$,
\[
\gamma^2\xi_a^2 - 2a\xi_a - \delta^2
= \xi_a(\gamma^2\xi_a - 2a) - \delta^2
= \xi_a\left(\sqrt{a^2 + D^2} - a\right) - \delta^2
= \frac{D^2}{\gamma^2} - \delta^2 = 0.
\]
Hence $\gamma^2\xi_a^2 - 2a\xi_a = \delta^2$, so we get
\[
(\xi_a A_b - a)^2
= a^2 + A_b(\xi_a^2A_b - 2a\xi_a)
= a^2 + A_b(\gamma^2\xi_a^2 - 2a\xi_a + b^2\xi_a^2)
= a^2 + A_b B_b.
\]
Since $\xi_aA_b - a = \xi_a(\gamma^2 + b^2) - a = \sqrt{a^2 + D^2} + b^2\xi_a > 0$, it follows that $\xi_aA_b = a + \sqrt{a^2 + x_b^2}$, and dividing this identity by $\sqrt{A_bB_b} = x_b$ gives
\begin{equation}\label{eq:GH.comparison.algebra}
\xi_a\sqrt{\frac{A_b}{B_b}} = \frac{a + \sqrt{a^2 + x_b^2}}{x_b}.
\end{equation}
Using \eqref{eq:GH.comparison.algebra}, the expression in square brackets in \eqref{eq:GH.comparison.Psi.prime} can be written as
\[
\xi_a - \sqrt{\frac{B_b}{A_b}}\frac{K_{\nu + 1}(x_b)}{K_{\nu}(x_b)}
= \sqrt{\frac{B_b}{A_b}}\left[\xi_a\sqrt{\frac{A_b}{B_b}} - \frac{K_{\nu + 1}(x_b)}{K_{\nu}(x_b)}\right]
= \sqrt{\frac{B_b}{A_b}}\left[\frac{a + \sqrt{a^2 + x_b^2}}{x_b} - \frac{K_{\nu + 1}(x_b)}{K_{\nu}(x_b)}\right].
\]
Since $\sqrt{B_b/A_b} > 0$, the sign of $\Psi_a'(b)$ coincides with the sign of
\[
\frac{a + \sqrt{a^2 + x_b^2}}{x_b} - \frac{K_{\nu + 1}(x_b)}{K_{\nu}(x_b)}.
\]
If \eqref{eq:GH.comparison.lower.condition} holds, then the last display is strictly negative at $x = x_b$. Thus, by \eqref{eq:GH.comparison.Psi.prime}, $\Psi_a'(b) < 0$ for all $b > 0$, and so $b\mapsto \Psi_a(b)$ is strictly decreasing on $(0,\infty)$. If the inequality in \eqref{eq:GH.comparison.lower.condition} is reversed for all $x > 0$, then the same argument gives $\Psi_a'(b) > 0$ for all $b > 0$, and hence $b\mapsto \Psi_a(b)$ is strictly increasing on $(0,\infty)$. Finally, if equality holds in \eqref{eq:GH.comparison.lower.condition} for all $x > 0$, then \eqref{eq:GH.comparison.Psi.prime} gives $\Psi_a'(b) = 0$ for all $b > 0$. Since $\Psi_a$ is continuous on $[0,\infty)$, it is constant. This completes the proof.
\end{proof}

\section{Proofs of the main results}\label{sec:proofs}

\subsection{Proof of Theorem~\ref{prop:conj3.2}}\label{subsec:proof.conj3.2}

Let $0 < \alpha < 1$ and define $a \leqdef 1 - \alpha\in(0,1)$. By the beta-gamma decomposition in \eqref{eq:definitions.positive.gamma.representation}, $Z_{\alpha} = S(1 + aV)$, where $S\sim \Gamma(2r,1)$ is independent of the symmetric random variable $V\in(-1,1)$. Conditioning on $|V|$ and applying Lemma~\ref{lem:gamma.reflection.bound} with $k = 2r$ and $u = a|V|$ gives
\[
\PP(Z_{\alpha} \leq 2r\mid |V|) = \frac{1}{2}G_{2r}\left(\frac{2r}{1 + a|V|}\right) + \frac{1}{2}G_{2r}\left(\frac{2r}{1 - a|V|}\right) > \frac{1}{2}.
\]
Taking expectations yields
\begin{equation}\label{eq:proof.conj3.2.median.mean}
\PP(Z_{\alpha} \leq 2r) > \frac{1}{2}, \qquad 0 < \alpha < 1.
\end{equation}
Consequently $\Med(Z_{\alpha}) \leq 2r$ for every $\alpha\in(0,1)$.

Now fix $0 < \alpha_2 < \alpha_1 < 1$ and set $m_1 \leqdef \Med(Z_{\alpha_1})$. By \eqref{eq:proof.conj3.2.median.mean}, $\PP(Z_{\alpha_1} \leq 2r) > 1/2$. Since the law of $Z_{\alpha_1}$ is continuous and has a strictly positive density on $(0,\infty)$, and since $Z_{\alpha_1} > 0$ almost surely, we have $0 < m_1 < 2r$ and $\PP(Z_{\alpha_1} \leq m_1) = 1/2$. Applying the strict Schur-convexity inequality \eqref{eq:Schur.special.case} from the proof of Lemma~\ref{lem:Schur.gamma} with $t = m_1$ gives
\[
\PP(Z_{\alpha_2} \leq m_1) > \PP(Z_{\alpha_1} \leq m_1) = \frac{1}{2}.
\]
Since the law of $Z_{\alpha_2}$ is continuous, this implies $\Med(Z_{\alpha_2}) < m_1 = \Med(Z_{\alpha_1})$. Therefore $\alpha\mapsto\Med(Z_{\alpha})$ is strictly increasing on $(0,1)$.

Finally, because $X_1$ and $X_2$ are independent and identically distributed, $Z_{\alpha} \, \smash{\stackrel{d}{=}} \, Z_{2 - \alpha}$. Hence $\Med(Z_{\alpha}) = \Med(Z_{2 - \alpha})$ for $0 < \alpha < 2$. The strict decrease on $(1,2)$ follows immediately from the strict increase on $(0,1)$ after replacing $\alpha$ by $2 - \alpha$.

\subsection{Proof of Theorem~\ref{prop:conj3.3}}\label{subsec:proof.conj3.3}

Let $a \leqdef \alpha - 1 > 1$ and $\tau \leqdef a^2 - 1 = \alpha(\alpha - 2)$. Since $\tau$ is strictly increasing as a function of $\alpha$ on $(2,\infty)$, it is enough to prove that the median is strictly decreasing as a function of $\tau\in(0,\infty)$. Let $G\sim \Gamma(r,1)$ and let $N\sim \mathcal{N}(0,1)$ be independent of $G$. We first note the distributional identity
\begin{equation}\label{eq:proof.conj3.3.normal.gamma.representation}
Z_{\alpha}\stackrel{d}{=}Y_{\tau} \leqdef 2G + \sqrt{2\tau G}N.
\end{equation}
To verify \eqref{eq:proof.conj3.3.normal.gamma.representation}, we note that, since $\alpha = 1 + \sqrt{1 + \tau}$ and $2-\alpha = 1-\sqrt{1 + \tau}$, we have $Z_{\alpha} = (1 + \sqrt{1 + \tau})X_1-(\sqrt{1 + \tau}-1)X_2$, where $X_1$ and $X_2$ are independent $\Gamma(r,1)$ random variables, and so by \cite[Eq.\ (20)]{FischerGauntSarantsev2025VGReview} it follows that $Z_{\alpha}\sim \mathrm{VG}(2r,1,\sqrt{\tau},0)$. By the normal variance-mean mixture representation of the VG distribution (see \cite[Eq.\ (18)]{FischerGauntSarantsev2025VGReview}) we obtain \eqref{eq:proof.conj3.3.normal.gamma.representation}.

Let $F_{\tau}$ denote the cdf of $Y_{\tau}$, and let $m(\tau) \leqdef \Med(Y_{\tau})$. The median is positive. To see this, we write $Z_{\alpha} = (1 + a)X_1 - (a - 1)X_2$ and use the beta-gamma decomposition. Then
\[
\PP(Z_{\alpha} \leq 0) = \PP\left(\frac{X_1}{X_1 + X_2} \leq \frac{a - 1}{2a}\right) < \frac{1}{2},
\]
because $X_1/(X_1 + X_2)\sim\operatorname{Beta}(r,r)$ is symmetric about $1/2$ and $(a - 1)/(2a) < 1/2$. Hence $m(\tau) > 0$.

Let $\Phi$ and $\phi$ denote the standard normal cdf and density, and let $g_k$ denote the density of the $\Gamma(k,1)$ distribution. For $q > 0$, conditioning on $G$ gives
\begin{equation}\label{gggg}
F_{\tau}(2q) = \int_0^{\infty} \Phi\left(\frac{2q - 2s}{\sqrt{2\tau s}}\right)g_r(s) \, \rd s.
\end{equation}
For $y > 0$, we define
\[
A_{\tau,q}(y) \leqdef \Phi(B_{\tau,q}(y)) - \frac{1}{2}, \qquad B_{\tau,q}(y) \leqdef 2\sqrt{\frac{2q}{\tau}}\sinh\left(\frac{y}{2}\right).
\]
We note that
\[
\frac{2q - 2qe^{-y}}{\sqrt{2\tau qe^{-y}}} = B_{\tau,q}(y), \qquad \frac{2q - 2qe^y}{\sqrt{2\tau qe^y}} = -B_{\tau,q}(y).
\]
Splitting the integral \eqref{gggg} at $s = q$ and making the changes of variables $s = qe^{-y}$ on $(0,q)$ and $s = qe^y$ on $(q,\infty)$ gives that
\[
\begin{aligned}
F_{\tau}(2q) &= \frac{q^r}{ \Gamma(r)}\int_0^{\infty} \Phi(B_{\tau,q}(y))e^{-r y - qe^{-y}} \, \rd y + \frac{q^r}{ \Gamma(r)}\int_0^{\infty} \Phi(-B_{\tau,q}(y))e^{r y - qe^y} \, \rd y \\
&= \frac{q^r}{ \Gamma(r)}\int_0^{\infty} \left(A_{\tau,q}(y) + \frac{1}{2}\right)e^{-r y - qe^{-y}} \, \rd y + \frac{q^r}{ \Gamma(r)}\int_0^{\infty} \left(\frac{1}{2} - A_{\tau,q}(y)\right)e^{r y - qe^y} \, \rd y,
\end{aligned}
\]
where we used $\Phi(-x) = 1 - \Phi(x)$ in the second equality. The terms containing $1/2$ add up to $1/2$, because
\[
\frac{q^r}{ \Gamma(r)}\int_0^{\infty} e^{-r y - qe^{-y}} \, \rd y + \frac{q^r}{ \Gamma(r)}\int_0^{\infty} e^{r y - qe^y} \, \rd y = \int_0^q g_r(s) \, \rd s + \int_q^{\infty} g_r(s) \, \rd s = 1.
\]
Consequently,
\begin{equation}\label{eq:proof.conj3.3.pairing.identity}
\begin{aligned}
F_{\tau}(2q) - \frac{1}{2}
&= \frac{q^r}{ \Gamma(r)}\int_0^{\infty} A_{\tau,q}(y)\left(e^{-r y - qe^{-y}} - e^{r y - qe^y}\right) \, \rd y \\
&= \frac{q^r}{ \Gamma(r)}\int_0^{\infty} A_{\tau,q}(y)\left(e^{-r y - q(\cosh(y) - \sinh(y))} - e^{r y - q(\cosh(y) + \sinh(y))}\right) \, \rd y \\
&= \frac{q^r}{ \Gamma(r)}\int_0^{\infty} A_{\tau,q}(y)e^{-q\cosh(y)}\left(e^{q\sinh(y) - r y} - e^{-q\sinh(y) + r y}\right) \, \rd y \\
&= \frac{2q^r}{ \Gamma(r)}\int_0^{\infty} A_{\tau,q}(y)e^{-q\cosh(y)}\sinh(q\sinh(y) - r y) \, \rd y.
\end{aligned}
\end{equation}
Since $A_{\tau,r}(y) > 0$ and $\sinh(y) > y$ for $y > 0$, we have $r(\sinh(y) - y) > 0$ and therefore $\sinh(r(\sinh(y) - y)) \geq r(\sinh(y) - y) > 0$. Thus \eqref{eq:proof.conj3.3.pairing.identity} with $q = r$ yields
\[
\begin{aligned}
F_{\tau}(2r)
&= \frac{1}{2} + \frac{2r^r}{ \Gamma(r)}\int_0^{\infty} A_{\tau,r}(y)e^{-r\cosh(y)}\sinh(r(\sinh(y) - y)) \, \rd y \\
&\geq \frac{1}{2} + \frac{2r^{r + 1}}{ \Gamma(r)}\int_0^{\infty} A_{\tau,r}(y)e^{-r\cosh(y)}(\sinh(y) - y) \, \rd y \\
&> \frac{1}{2}.
\end{aligned}
\]
Together with $F_{\tau}(0) < 1/2$, this proves
\[
0 < m(\tau) < 2r.
\]
Set
\[
q_{\tau} \leqdef \frac{m(\tau)}{2}.
\]
Then $0 < q_{\tau} < r$, and \eqref{eq:proof.conj3.3.pairing.identity} at $q = q_{\tau}$ gives
\begin{equation}\label{eq:proof.conj3.3.median.identity}
\int_0^{\infty} A_{\tau,q_{\tau}}(y)L_{q_{\tau}}(y) \, \rd y = 0,
\end{equation}
where
\[
L_q(y) \leqdef e^{-q\cosh(y)}\sinh(q\sinh(y) - r y).
\]

We next record the sign pattern of $L_q$ when $0 < q < r$. The function $y\mapsto q\sinh(y) - r y$ is negative for all sufficiently small positive $y$, tends to infinity as $y\to\infty$, and has exactly one zero in $(0,\infty)$. Indeed, its derivative is $q\cosh(y) - r$, which is strictly increasing on $(0,\infty)$ and changes sign exactly once. Therefore $L_q$ changes sign exactly once on $(0,\infty)$, and the change is from negative to positive.

Differentiating \eqref{eq:proof.conj3.3.pairing.identity} with respect to $\tau$ at fixed $q$ gives
\begin{equation}\label{eq:proof.conj3.3.tau.derivative}
\frac{\partial}{\partial\tau}F_{\tau}(2q) = -\frac{q^r}{\tau \Gamma(r)}\int_0^{\infty} B_{\tau,q}(y)\phi(B_{\tau,q}(y))L_q(y) \, \rd y.
\end{equation}
The differentiation is justified by dominated convergence locally uniformly for $\tau,q > 0$, because the factor $e^{-q\cosh(y)}$ gives exponential integrability as $y\to\infty$, while $A_{\tau,q}(y)$ and $B_{\tau,q}(y)\phi(B_{\tau,q}(y))$ are both of order $y$ as $y\downarrow0$.

We now compare the two weights $A_{\tau,q}$ and $B_{\tau,q}\phi(B_{\tau,q})$. Define
\[
\rho(x) \leqdef \frac{x\phi(x)}{\Phi(x) - 1/2}, \qquad x > 0.
\]
We claim that $\rho$ is strictly decreasing on $(0,\infty)$. Let $H(x) \leqdef \Phi(x) - 1/2$. Then
\[
\rho'(x) = \frac{\phi(x)}{H(x)^2}\left\{(1 - x^2)H(x) - x\phi(x)\right\}.
\]
To see that the expression in braces is negative, let $J(x) \leqdef x\phi(x) - (1 - x^2)H(x)$. Then $J(0) = 0$ and $J'(x) = 2xH(x) > 0$, $x > 0$. Thus $J(x) > 0$ for $x > 0$, proving that $\rho'(x) < 0$.

Since $y\mapsto B_{\tau,q_{\tau}}(y)$ is strictly increasing on $(0,\infty)$, the function $y\mapsto \rho(B_{\tau,q_{\tau}}(y))$ is strictly decreasing. Let $y_{\tau}$ be the unique positive zero of $L_{q_{\tau}}$. Recall that from \eqref{eq:proof.conj3.3.median.identity} it holds that $\int_0^{\infty} A_{\tau,q_{\tau}}(y)L_{q_{\tau}}(y) \, \rd y = 0$.
Using also that $B_{\tau,q_{\tau}}(y)\phi(B_{\tau,q_{\tau}}(y)) = \rho(B_{\tau,q_{\tau}}(y))A_{\tau,q_{\tau}}(y)$, we get
\[
\begin{aligned}
\int_0^{\infty} B_{\tau,q_{\tau}}(y)\phi(B_{\tau,q_{\tau}}(y))L_{q_{\tau}}(y) \, \rd y &= \int_0^{\infty} \rho(B_{\tau,q_{\tau}}(y))A_{\tau,q_{\tau}}(y)L_{q_{\tau}}(y) \, \rd y \\
&= \int_0^{\infty} \{\rho(B_{\tau,q_{\tau}}(y)) - \rho(B_{\tau,q_{\tau}}(y_{\tau}))\}A_{\tau,q_{\tau}}(y)L_{q_{\tau}}(y) \, \rd y \\
& < 0.
\end{aligned}
\]
The last inequality follows because $L_{q_{\tau}}(y) < 0$ on $(0,y_{\tau})$ and $L_{q_{\tau}}(y) > 0$ on $(y_{\tau},\infty)$, while $\rho(B_{\tau,q_{\tau}}(y)) - \rho(B_{\tau,q_{\tau}}(y_{\tau}))$ is positive on $(0,y_{\tau})$ and negative on $(y_{\tau},\infty)$. Therefore \eqref{eq:proof.conj3.3.tau.derivative} gives
\[
\left.\frac{\partial}{\partial\tau}F_{\tau}(t)\right|_{t = m(\tau)} > 0.
\]

Finally, $F_{\tau}$ is continuously differentiable in both variables $t > 0$ and $\tau > 0$, and its density is strictly positive at $m(\tau) > 0$. Hence the implicit function theorem applied to $F_{\tau}(m(\tau)) = 1/2$ yields
\[
m'(\tau) = -\frac{\left.\frac{\partial}{\partial\tau}F_{\tau}(t)\right|_{t = m(\tau)}}{\left.\frac{\partial}{\partial t}F_{\tau}(t)\right|_{t = m(\tau)}} < 0, \qquad \tau > 0.
\]
Thus $\tau\mapsto m(\tau)$ is strictly decreasing on $(0,\infty)$. Since $\tau = \alpha(\alpha - 2)$ is strictly increasing on $(2,\infty)$, the function $\alpha\mapsto \Med(Z_{\alpha})$ is strictly decreasing on $(2,\infty)$. This completes the proof.

\subsection{Proof of Corollary~\ref{cor:conj3.5}}\label{subsec:proof.conj3.5}

Let $s \leqdef r/2$ and $\kappa \leqdef \sigma^2/\theta^2$. By \eqref{eq:definitions.VG.gamma.representation}, $V_{r,\theta,\sigma}\stackrel{d}{=}\theta Z_{\alpha_{\sigma}}$ with $\alpha_{\sigma} \leqdef 1 + \sqrt{1 + \kappa}$, where $Z_{\alpha_{\sigma}}$ is a linear combination of two independent $\Gamma(s,1)$ random variables. Since $\sigma\mapsto\alpha_{\sigma}$ is strictly increasing from $(0,\infty)$ onto $(2,\infty)$, Theorem~\ref{prop:conj3.3} proves that $\sigma\mapsto\Med(V_{r,\theta,\sigma})$ is strictly decreasing. Letting $\sigma\downarrow0$ gives $\alpha_{\sigma}\downarrow2$. Also, $\theta Z_2 = 2\theta Y_1\sim \Gamma(r/2,(2\theta)^{-1})$. Since the cdf of $G$ is continuous and strictly increasing at its median, its quantile function is continuous at $1/2$. Therefore, the convergence of quantiles under weak convergence \cite[Lemma~21.2]{vdv98} gives
\[
\lim_{\sigma\downarrow0}\Med(V_{r,\theta,\sigma}) = \Med(G),
\]
and hence $\Med(V_{r,\theta,\sigma}) < \Med(G)$ for every $\sigma > 0$. This proves the claims \eqref{eq:cor:conj3.5.claim.1} and \eqref{eq:cor:conj3.5.claim.2}.

It remains to compute the limit as $\sigma\to\infty$. Let $G_s\sim \Gamma(s,1)$ and let $N\sim \mathcal{N}(0,1)$ be independent of $G_s$. By the representation in \eqref{eq:proof.conj3.3.normal.gamma.representation}, applied with shape parameter~$s$, the median of $Z_{\alpha_{\sigma}}$ is the median of
\[
Y_{\kappa} \leqdef 2G_s + \sqrt{2\kappa G_s}N.
\]
Let $F_{\kappa}$ be the cdf of $Y_{\kappa}$ and let $m_s(\kappa) \leqdef \Med(Y_{\kappa})$. The proof of Theorem~\ref{prop:conj3.3}, with $r$ replaced by $s$, gives $0 < m_s(\kappa) < 2s$ and the identity
\begin{equation}\label{eq:proof.conj3.5.pairing.identity}
F_{\kappa}(2q) - \frac{1}{2} = \frac{2q^s}{ \Gamma(s)}\int_0^{\infty} A_{\kappa,q}(y)e^{-q\cosh(y)}\sinh(q\sinh(y) - s y) \, \rd y,
\end{equation}
where
\begin{equation}\label{eq:A.def}
A_{\kappa,q}(y) \leqdef \Phi\left(2\sqrt{\frac{2q}{\kappa}}\sinh\left(\frac{y}{2}\right)\right) - \frac{1}{2}.
\end{equation}

Assume first that $s > 1/2$. For fixed $q > 0$ and $y > 0$, since $\Phi(x) = 1/2 + x/\sqrt{2\pi} + o(x)$ as $x\to0$, we have
\[
\sqrt{\kappa}A_{\kappa,q}(y)\to 2\sqrt{\frac{q}{\pi}}\sinh\left(\frac{y}{2}\right), \qquad \kappa\to\infty.
\]
Moreover, the convergence is dominated by an integrable function. Indeed, the elementary bound $\Phi(x) - 1/2 \leq x/\sqrt{2\pi}$, $x \geq 0$, gives
\begin{equation}\label{eq:A.bound}
0 \leq \sqrt{\kappa}A_{\kappa,q}(y) \leq 2\sqrt{\frac{q}{\pi}}\sinh\left(\frac{y}{2}\right), \qquad y > 0.
\end{equation}
Therefore the absolute value of the integrand in $\sqrt{\kappa}$ times the right-hand side of \eqref{eq:proof.conj3.5.pairing.identity} is bounded, up to a constant depending only on $q$ and $s$, by
\[
\sinh\left(\frac{y}{2}\right)e^{-q\cosh(y)}|\sinh(q\sinh(y) - s y)|.
\]
This function is integrable on $(0,\infty)$. Near the origin it is of order $y^2$. As $y\to\infty$, using $|\sinh(y)| \leq (e^y+e^{-y})/2$, $\cosh(y) - \sinh(y) = e^{-y}$ and $\cosh(y) + \sinh(y) = e^y$, we have
\[
\begin{aligned}
\sinh\left(\frac{y}{2}\right)e^{-q\cosh(y)}|\sinh(q\sinh(y) - s y)|
&\leq \frac{1}{2}\sinh\left(\frac{y}{2}\right)\left(e^{-q e^{-y} - s y} + e^{-q e^y + s y}\right) \\
&\leq C_{q,s}\left(e^{-(s - 1/2)y} + e^{-q e^y + (s + 1/2)y}\right),
\end{aligned}
\]
which is integrable precisely because $s > 1/2$. Hence dominated convergence applies in \eqref{eq:proof.conj3.5.pairing.identity}, and the sign of $F_{\kappa}(2q) - 1/2$ for all sufficiently large $\kappa$ is the sign of
\[
I_s(q) \leqdef \int_0^{\infty} \sinh\left(\frac{y}{2}\right)e^{-q\cosh(y)}\sinh(q\sinh(y) - s y) \, \rd y.
\]
Using $e^{-q\cosh(y)}\sinh(q\sinh(y) - s y) = (e^{-qe^{-y} - sy} - e^{-qe^y + sy})/2$, and then the changes of variables $t = e^{-y}$ and $t = e^y$ in the two terms, we obtain
\[
\begin{aligned}
I_s(q)
&= \frac{1}{4}\int_0^{\infty} (t^{s - 3/2} - t^{s - 1/2})e^{-qt} \, \rd t \\
&= \frac{1}{4}\left(\frac{ \Gamma(s - 1/2)}{q^{s - 1/2}} - \frac{ \Gamma(s + 1/2)}{q^{s + 1/2}}\right)
= \frac{ \Gamma(s - 1/2)}{4q^{s + 1/2}}\left(q - s + \frac{1}{2}\right).
\end{aligned}
\]
Therefore, if $q > s - 1/2$, then $F_{\kappa}(2q) > 1/2$ for all sufficiently large $\kappa$, while if $0 < q < s - 1/2$, then $F_{\kappa}(2q) < 1/2$ for all sufficiently large $\kappa$. Hence
\begin{equation}\label{eq:proof.conj3.5.limit.positive}
\lim_{\kappa\to\infty}m_s(\kappa) = 2s - 1, \qquad s > \frac{1}{2}.
\end{equation}

Now assume that $0 < s \leq 1/2$. Fix $q > 0$. We prove that $F_{\kappa}(2q) > 1/2$ for all sufficiently large $\kappa$. Let $L_q(y)$ be defined as in the proof of Theorem~\ref{prop:conj3.3}. Using $\cosh(y) - \sinh(y) = e^{-y}$ and $\cosh(y) + \sinh(y) = e^y$, we can rewrite $L_q(y)$ as
\[
\begin{aligned}
L_q(y)
&= \frac{1}{2}e^{-q\cosh(y)}\left(e^{q\sinh(y) - s y} - e^{-q\sinh(y) + s y}\right) \\
&= \frac{1}{2}\left(e^{-s y - qe^{-y}} - e^{s y - qe^y}\right)
= \frac{1}{2}e^{-s y - qe^{-y}}\left(1 - e^{2s y - q(e^y - e^{-y})}\right).
\end{aligned}
\]
Since $q(e^y - e^{-y}) - 2s y\to\infty$ as $y\to\infty$, there is $y_0 > 0$ such that $e^{2s y - q(e^y - e^{-y})} \leq 1/2$ for $y \geq y_0$. Hence, for $y \geq y_0$,
\[
L_q(y) \geq \frac{1}{4}e^{-s y - qe^{-y}} \geq \frac{e^{-q}}{4}e^{-s y}.
\]
Thus $L_q(y) > 0$ for $y \geq y_0$, and the lower bound holds with $c_1 \leqdef e^{-q}/4$. On the interval $(0,y_0)$, the function $\sinh(y/2)|L_q(y)|$ is integrable. Indeed, it is continuous on every interval $[\eta,y_0]$, $\eta > 0$, and as $y\downarrow0$ we have $\sinh(y/2) = O(y)$ and $q\sinh(y) - s y = (q - s)y + O(y^3)$,
so that $L_q(y) = O(y)$ as $y\downarrow0$. Thus $\sinh(y/2)|L_q(y)| = O(y^2)$ near the origin. By \eqref{eq:A.bound}, there is a constant $c_2 > 0$, depending on $q$, such that
\[
0 \leq A_{\kappa,q}(y) \leq c_2\kappa^{-1/2}\sinh\left(\frac{y}{2}\right), \qquad y > 0.
\]
Therefore the contribution to the integral in \eqref{eq:proof.conj3.5.pairing.identity} over the interval $(0,y_0)$ is bounded below by
\[
\int_0^{y_0}A_{\kappa,q}(y)L_q(y) \, \rd y \geq -c_2\kappa^{-1/2}\int_0^{y_0}\sinh\left(\frac{y}{2}\right)|L_q(y)| \, \rd y \geq -c_3\kappa^{-1/2},
\]
for some constant $c_3 > 0$. This estimate controls all possible negative contributions to the integral in \eqref{eq:proof.conj3.5.pairing.identity}, because we shall only use intervals contained in $[y_0,\infty)$ for the positive contribution.

Assume first that $0 < s < 1/2$. For $y\in[\log(\kappa),\log(\kappa) + 1]$ and $\kappa$ sufficiently large, this interval is contained in $[y_0,\infty)$, and
\[
2\sqrt{\frac{2q}{\kappa}}\sinh\left(\frac{y}{2}\right)
\geq 2\sqrt{\frac{2q}{\kappa}}\sinh\left(\frac{\log(\kappa)}{2}\right)
= 2\sqrt{\frac{2q}{\kappa}}\cdot\frac{\sqrt{\kappa} - 1/\sqrt{\kappa}}{2}
= \sqrt{2q}\left(1 - \frac{1}{\kappa}\right).
\]
Hence the argument of $\Phi$ in $A_{\kappa,q}$ (recall \eqref{eq:A.def}) is bounded below by a positive constant depending only on $q$, and so there exists $c_4 > 0$ such that
\[
A_{\kappa,q}(y) \geq c_4, \qquad \log(\kappa) \leq y \leq \log(\kappa) + 1.
\]
Using the lower bound $L_q(y) \geq c_1e^{-sy}$ on $[y_0,\infty)$, we obtain
\[
\int_{\log(\kappa)}^{\log(\kappa) + 1}A_{\kappa,q}(y)L_q(y) \, \rd y
\geq c_1c_4\int_{\log(\kappa)}^{\log(\kappa) + 1}e^{-sy} \, \rd y
= c_1c_4\kappa^{-s}\int_0^1e^{-su} \, \rd u
\geq c_5\kappa^{-s},
\]
for some constant $c_5 > 0$. Since $s < 1/2$, the term $\kappa^{-s}$ dominates $\kappa^{-1/2}$ as $\kappa\to\infty$. Combining this positive contribution with the lower bound $-c_3\kappa^{-1/2}$ from $(0,y_0)$ shows that the integral in \eqref{eq:proof.conj3.5.pairing.identity} is strictly positive for all sufficiently large $\kappa$.

It remains to treat the borderline case $s = 1/2$. For $\kappa$ sufficiently large, the compact interval $[y_0,(1/2)\log(\kappa)]$ is non-empty. On this interval, the argument of $\Phi$ in $A_{\kappa,q}$ is bounded above, and in fact tends to zero uniformly as $\kappa\to\infty$. Moreover, since $y \geq y_0$, there is a constant $c_6 > 0$ such that
\[
2\sqrt{\frac{2q}{\kappa}}\sinh\left(\frac{y}{2}\right) \geq c_6\kappa^{-1/2}e^{y/2}, \qquad y_0 \leq y \leq \frac{1}{2}\log(\kappa).
\]
Using the elementary lower bound $\Phi(x) - 1/2 \geq c_7x$ for $0 \leq x \leq C$ on a fixed bounded interval, we get, for all sufficiently large $\kappa$,
\[
A_{\kappa,q}(y) \geq c_8\kappa^{-1/2}e^{y/2}, \qquad y_0 \leq y \leq \frac{1}{2}\log(\kappa).
\]
Since $s = 1/2$, the lower bound $L_q(y) \geq c_1e^{-sy}$ becomes $L_q(y) \geq c_1e^{-y/2}$ on $[y_0,\infty)$. Therefore
\[
\begin{aligned}
\int_{y_0}^{(1/2)\log(\kappa)}A_{\kappa,q}(y)L_q(y) \, \rd y & \geq c_1c_8\kappa^{-1/2}\int_{y_0}^{(1/2)\log(\kappa)}e^{y/2}e^{-y/2} \, \rd y \\
&= c_1c_8\kappa^{-1/2}\left(\frac{1}{2}\log(\kappa) - y_0\right)
\geq c_9\kappa^{-1/2}\log(\kappa),
\end{aligned}
\]
for some constant $c_9 > 0$ and all sufficiently large $\kappa$. This dominates the possible negative contribution $c_3\kappa^{-1/2}$ from $(0,y_0)$. Hence, also when $s = 1/2$, the integral in \eqref{eq:proof.conj3.5.pairing.identity} is strictly positive for all sufficiently large $\kappa$.

Therefore, we have proved that, for every fixed $q > 0$ and every $0 < s \leq 1/2$, we have
$F_{\kappa}(2q) > 1/2$
for all sufficiently large $\kappa$. Thus $m_s(\kappa) \leq 2q$ for all sufficiently large $\kappa$. Since $q > 0$ was arbitrary and $m_s(\kappa) > 0$, it follows that
\begin{equation}\label{eq:proof.conj3.5.limit.zero}
\lim_{\kappa\to\infty}m_s(\kappa) = 0, \qquad 0 < s \leq \frac{1}{2}.
\end{equation}
Combining \eqref{eq:proof.conj3.5.limit.positive} and \eqref{eq:proof.conj3.5.limit.zero}, and multiplying by $\theta$, gives the asserted limit \eqref{eq:asserted.Med} for $\Med(V_{r,\theta,\sigma})$.

It remains to derive the displayed bounds \eqref{eq:results.VG.bounds.general} and \eqref{eq:results.VG.bounds.refined}. By \eqref{eq:asserted.Med} and the strict decrease of $\sigma\mapsto\Med(V_{r,\theta,\sigma})$, we have
\[
\Med(V_{r,\theta,\sigma}) > (r - 1)\theta, \qquad r > 1.
\]
Together with the positivity of $\Med(V_{r,\theta,\sigma})$ for all $r > 0$ we obtain the lower bound in \eqref{eq:results.VG.bounds.general}.
The upper bound in \eqref{eq:results.VG.bounds.general} follows from \eqref{eq:cor:conj3.5.claim.2} and the gamma median upper bound recorded in \eqref{twoside1} applied to $G\sim \Gamma(r/2,(2\theta)^{-1})$. Finally, assume that $r \geq 2$. Then the upper bound in \eqref{eq:results.VG.bounds.refined} follows from an application of the gamma median upper bound from \eqref{twoside2} applied to $G\sim \Gamma(r/2,(2\theta)^{-1})$.

\subsection{Proof of Corollary~\ref{cor:conj3.6}}\label{subsec:proof.conj3.6}

Let $\alpha_c \leqdef (c-1)/c\in (0,1)$. By the McKay Type I representation \eqref{eq:definitions.McKay.gamma.representation}, with $Y_1,Y_2$ independent $\Gamma(m + 1/2,1)$ random variables,
\[
Z_{m,c,\phi}
= \frac{\phi(c - 1)}{c}Y_1 + \frac{\phi(c + 1)}{c}Y_2
= \phi\left(\alpha_cY_1 + (2 - \alpha_c)Y_2\right)
\stackrel{d}{=} \phi Z_{\alpha_c}.
\]
Since $c\mapsto\alpha_c$ is strictly increasing from $(1,\infty)$ onto $(0,1)$, Theorem~\ref{prop:conj3.2} gives the strict increase of $c\mapsto\Med(Z_{m,c,\phi})$. Letting $c\downarrow1$ gives $\alpha_c\downarrow0$ and hence $\phi Z_{\alpha_c} \, \smash{\stackrel{d}{\longrightarrow}} \, 2\phi Y_2\sim \Gamma(m + 1/2,(2\phi)^{-1})$, while letting $c\to\infty$ gives $\alpha_c\uparrow1$ and hence $\phi Z_{\alpha_c} \, \smash{\stackrel{d}{\longrightarrow}} \, \phi(Y_1 + Y_2)\sim \Gamma(2m + 1,1/\phi)$. Since the cdfs of the two limiting gamma distributions are continuous and strictly increasing at their medians, their quantile functions are continuous at $1/2$. Therefore, the convergence of quantiles under weak convergence \cite[Lemma~21.2]{vdv98} gives
\[
\Med(G_1) < \Med(Z_{m,c,\phi}) < \Med(G_2), \qquad c > 1,
\]
where $G_1\sim \Gamma(m + 1/2,(2\phi)^{-1})$ and $G_2\sim \Gamma(2m + 1,1/\phi)$.

The general lower and upper bounds in \eqref{eq:results.McKay.bounds.general} follow from the gamma median bounds recorded in \eqref{twoside1}, applied to $G_1\sim \Gamma(m + 1/2,(2\phi)^{-1})$ and $G_2\sim \Gamma(2m + 1,1/\phi)$, respectively. The refined lower and upper bounds \eqref{eq:results.McKay.bounds.refined} for $m \geq 1/2$ similarly follow from the bounds for the median of the gamma distribution given in \eqref{twoside2} applied to $G_1\sim \Gamma(m + 1/2,(2\phi)^{-1})$ and $G_2\sim \Gamma(2m + 1,1/\phi)$, respectively.

\subsection{Proof of Theorem~\ref{prop:conj3.7}}\label{subsec:proof.conj3.7}

Assume first that $\delta > 0$. Set $\nu \leqdef \lambda - 1/2$, $a \leqdef \nu$ and $\xi_a \leqdef (a + \sqrt{a^2 + \delta^2\gamma^2})/\gamma^2$.
By the lower bound in \eqref{eq:Bessel.ratio.two.sided}, the condition \eqref{eq:GH.comparison.lower.condition} of Lemma~\ref{lem:GH.comparison} is satisfied with this choice of $a$. Hence, $b\mapsto \Psi_a(b) = \PP(b W + \sqrt{W} N \leq b \xi_a)$ is strictly decreasing on $(0,\infty)$. Since $\Psi_a(0) = 1/2$ and $\beta > 0$, substituting $b = \beta$ and using the mixture representation \eqref{eq:definitions.GH.mixture} gives
\[
\PP(X \leq \beta\xi_a) = \Psi_a(\beta) < \Psi_a(0) = \frac{1}{2}.
\]
Therefore $\Med(X) > \beta\xi_a$, which is the desired lower bound.

It remains only to consider the limiting variance-gamma case $\delta = 0$, for which the assumption is $\lambda > 1/2$. Using the parameter correspondence $r = 2\lambda$ and $\theta = \beta/\gamma^2$, the lower bound in \eqref{eq:results.VG.bounds.general} gives
\[
\Med(X) > (r - 1)\theta = \frac{\beta}{\gamma^2}(2\lambda - 1) = \frac{\beta}{\gamma^2}\left[\lambda - 1/2 + \sqrt{(\lambda - 1/2)^2}\right],
\]
which is \eqref{eq:results.conj3.7.bound} with $\delta = 0$.

\subsection{Proof of Theorem~\ref{prop:GH.upper.bounds}}\label{subsec:proof.GH.upper.bounds}

Set $\nu$ and $\xi_a$ as in the proof of Theorem \ref{prop:conj3.7}. We shall use without further mention that, for $\delta > 0$, the law of $X$ has a continuous strictly positive density on $\R$. Taking $a = \nu + 1 = \lambda + 1/2$ in Lemma~\ref{lem:GH.comparison}, the upper bound in \eqref{eq:Bessel.ratio.two.sided} shows that $b\mapsto \Psi_a(b) = \PP(b W + \sqrt{W} N \leq b \xi_a)$ is strictly increasing on $(0,\infty)$. Since $\Psi_a$ is continuous at zero, $\Psi_a(0) = 1/2$ and $\beta > 0$, substituting $b = \beta$ and using the mixture representation \eqref{eq:definitions.GH.mixture} gives
\[
\PP(X \leq \beta\xi_a) = \Psi_a(\beta) > \Psi_a(0) = \frac{1}{2},
\]
and therefore $\Med(X) < \beta\xi_a$. This is exactly \eqref{eq:results.GH.upper.global}, because $\beta\xi_a = u_{-1/2}$.

The assertions involving $u_0$ follow in the same way by taking $a = \nu + 1/2 = \lambda$ in Lemma~\ref{lem:GH.comparison} and using \eqref{eq:Bessel.ratio.half}. If $\lambda > 0$, then $\nu > -1/2$, so $b\mapsto \Psi_a(b)$ is strictly increasing. Since $\Psi_a$ is continuous at zero, substituting $b = \beta$ gives $\PP(X \leq u_0) = \Psi_a(\beta) > \Psi_a(0) = 1/2$, and hence $\Med(X) < u_0$. If $\lambda = 0$, then $\nu = -1/2$, equality holds in \eqref{eq:Bessel.ratio.half}, and $\Psi_a$ is constant. Consequently $\PP(X \leq u_0) = \Psi_a(\beta) = \Psi_a(0) = 1/2$, and hence $\Med(X) = u_0$. If $\lambda < 0$, then $\nu < -1/2$, the inequality in \eqref{eq:Bessel.ratio.half} is reversed, and $b\mapsto \Psi_a(b)$ is strictly decreasing. Since $\Psi_a$ is continuous at zero, substituting $b = \beta$ gives $\PP(X \leq u_0) = \Psi_a(\beta) < \Psi_a(0) = 1/2$, and hence $\Med(X) > u_0$.

It remains to justify the optimality assertions. Let $c_0 > -1/2$. We shall show that the upper bound $\Med(X) < u_{c_0}$ cannot hold uniformly over all $\lambda\in\R$ and $\delta > 0$. Choose $\lambda < -1/2$, set $\nu \leqdef \lambda - 1/2$ and $a_0 \leqdef \lambda - c_0$. Then $\nu < -1$ and $a_0 < 0$. Let $D \leqdef \delta\gamma$. By the small-argument expansion of the modified Bessel function of the second kind \cite[Eq.\ 10.30.2]{NISTDLMF} and the standard connection formula $K_{\lambda}(x) = K_{-\lambda}(x)$ \cite[Eq.\ 10.27.3]{NISTDLMF}, we have, as $D\downarrow0$,
\[
K_{\nu + 1}(D) = K_{-\nu-1}(D) \sim 2^{-\nu-2} \Gamma(-\nu - 1)D^{\nu + 1}, \qquad K_{\nu}(D) = K_{-\nu}(D) \sim 2^{-\nu-1} \Gamma(-\nu)D^{\nu},
\]
and, since $a_0 = \lambda - c_0 < 0$,
\[
a_0 + \sqrt{a_0^2 + D^2} = \frac{D^2}{\sqrt{a_0^2 + D^2} - a_0} \sim \frac{D^2}{-2a_0} = \frac{D^2}{2(c_0 - \lambda)}.
\]
Therefore, as $D\downarrow0$,
\[
\frac{K_{\nu + 1}(D)}{K_{\nu}(D)} = \frac{K_{-\nu-1}(D)}{K_{-\nu}(D)} \sim \frac{D}{-2\lambda - 1}, \qquad \frac{a_0 + \sqrt{a_0^2 + D^2}}{D} \sim \frac{D}{2(c_0 - \lambda)}.
\]
Since $c_0 > -1/2$, we have $2(c_0 - \lambda) > -2\lambda - 1$, and therefore
\[
\frac{a_0 + \sqrt{a_0^2 + D^2}}{D} < \frac{K_{\nu + 1}(D)}{K_{\nu}(D)}
\]
for all sufficiently small $D > 0$. Fix such a value of $D$, take $\gamma = 1$ and $\delta = D$, and then choose $\alpha = \sqrt{1 + \beta^2}$, so that $\sqrt{\alpha^2 - \beta^2} = \gamma$ for every $\beta > 0$. Equation \eqref{eq:GH.comparison.Psi.prime} from the proof of Lemma~\ref{lem:GH.comparison}, together with \eqref{eq:GH.comparison.algebra}, can also be used to evaluate the right derivative at $b = 0$ because $\delta > 0$. For the candidate bound with $c = c_0$, that is, with $a_0 \leqdef \lambda - c_0$, we obtain
\[
\Psi_{a_0}'(0 + ) = C_0\left[\frac{a_0 + \sqrt{a_0^2 + D^2}}{D} - \frac{K_{\nu + 1}(D)}{K_{\nu}(D)}\right],
\]
where $C_0$ is a positive constant. The preceding asymptotic comparison shows that the expression in square brackets is negative for all sufficiently small $D > 0$, and hence $\Psi_{a_0}'(0 + ) < 0$. Thus, for all sufficiently small $\beta > 0$,
\[
\PP(X \leq u_{c_0}) = \Psi_{a_0}(\beta) < \Psi_{a_0}(0) = \frac{1}{2}.
\]
For these parameter values the median is larger than $u_{c_0}$, so no $c_0 > -1/2$ can be valid uniformly over all $\lambda\in\R$ and $\delta > 0$.

Finally, let $c_0 > 0$ and fix $\lambda \geq 1$. Again, set $\nu = \lambda - 1/2$ and $a_0 = \lambda - c_0$. Let $D \leqdef \delta\gamma$. By the large-argument expansion of the modified Bessel function of the second kind \cite[Eq.\ 10.40.2]{NISTDLMF},
\[
K_{\lambda}(x) = \sqrt{\frac{\pi}{2x}}e^{-x}\left(1 + \frac{4\lambda^2 - 1}{8x} + O(x^{-2})\right), \qquad x\to\infty,
\]
and hence, as $D\to\infty$,
\[
\frac{K_{\nu + 1}(D)}{K_{\nu}(D)} = 1 + \frac{2\nu + 1}{2D} + O(D^{-2}) = 1 + \frac{\lambda}{D} + O(D^{-2}).
\]
Moreover, $\sqrt{a_0^2 + D^2} = D + a_0^2/(2D) + O(D^{-3})$, and therefore, as $D\to\infty$,
\[
\frac{a_0 + \sqrt{a_0^2 + D^2}}{D} = 1 + \frac{a_0}{D} + O(D^{-2}) = 1 + \frac{\lambda - c_0}{D} + O(D^{-2}).
\]
Since $c_0 > 0$, it follows that
\[
\frac{a_0 + \sqrt{a_0^2 + D^2}}{D} < \frac{K_{\nu + 1}(D)}{K_{\nu}(D)}
\]
for all sufficiently large $D > 0$. Fix such a value of $D$, take $\gamma = 1$ and $\delta = D$, and choose $\alpha = \sqrt{1 + \beta^2}$, so that $\sqrt{\alpha^2 - \beta^2} = \gamma$ for every $\beta > 0$. Equation \eqref{eq:GH.comparison.Psi.prime}, together with \eqref{eq:GH.comparison.algebra}, can also be used to evaluate the right derivative at $b = 0$ because $\delta > 0$. For the candidate bound with $c = c_0$, that is, with $a_0 \leqdef \lambda - c_0$, we obtain
\[
\Psi_{a_0}'(0 + ) = C_0\left[\frac{a_0 + \sqrt{a_0^2 + D^2}}{D} - \frac{K_{\nu + 1}(D)}{K_{\nu}(D)}\right],
\]
where $C_0$ is a positive constant. The preceding asymptotic comparison shows that the expression in square brackets is negative for all sufficiently large $D > 0$, and hence $\Psi_{a_0}'(0 + ) < 0$. Thus, for all sufficiently small $\beta > 0$,
\[
\PP(X \leq u_{c_0}) = \Psi_{a_0}(\beta) < \Psi_{a_0}(0) = \frac{1}{2}.
\]
For these parameter values the median is larger than $u_{c_0}$, so no $c_0 > 0$ can be valid uniformly over the range $\lambda \geq 1$ and $\delta > 0$. This proves that $c = 0$ is the largest uniform constant on the range $\lambda \geq 1$, and completes the proof.

\subsection{Proof of Corollary~\ref{mmm}}

\noindent 1. Corollary 2.6 of \cite{MR4141494} asserts that $\mathrm{Mode}(V_{r,\theta,\sigma}) = 0$ for $0 < r \leq 2$ and $\mathrm{Mode}(V_{r,\theta,\sigma}) < \theta(r - 2)$ for $r > 2$, whilst it is standard that $\EE[V_{r,\theta,\sigma}] = r\theta$ (see, for example, \cite[Eq.\ (30)]{FischerGauntSarantsev2025VGReview}). Combining these facts with the two-sided inequality \eqref{eq:results.VG.bounds.general} for $\Med(V_{r,\theta,\sigma})$ now yields inequality \eqref{mmm1}.

\vspace{3mm}

\noindent 2. Theorem 2.9 of \cite{MR4141494} states that $\mathrm{Mode}(Z_{m,c,b}) = 0$ for $-1/2 < m \leq 0$ and $\mathrm{Mode}(Z_{m,c,b}) < 2m\phi$ for $m > 0$, where $\phi = bc/(c^2 - 1)$, whilst the mean is given by $\EE[Z_{m,c,b}] = (2m + 1)\phi$ (see \cite{m32}). The inequality $\Med(Z_{m,c,b}) < \EE[Z_{m,c,b}]$ now follows from the upper bound of \eqref{eq:results.McKay.bounds.general}. Since $Z_{m,c,b} > 0$ almost surely, we have $\Med(Z_{m,c,b}) > 0$, and therefore $\mathrm{Mode}(Z_{m,c,b}) = 0 < \Med(Z_{m,c,b})$ for $-1/2 < m \leq 0$. Also, the inequality $\mathrm{Mode}(Z_{m,c,b}) < \Med(Z_{m,c,b})$ follows, for $m \geq 1/2$, from the lower bound in \eqref{eq:results.McKay.bounds.refined}.

To complete the proof, it remains to prove that $\mathrm{Mode}(Z_{m,c,b}) < \Med(Z_{m,c,b})$ for $0 < m < 1/2$. By the bound $(2m + 1)\phi e^{-2\log(2)/(2m + 1)} < \Med(Z_{m,c,b})$ from \eqref{eq:results.McKay.bounds.general} and the aforementioned mode bound, it is enough to prove that $(2m + 1) e^{-2\log(2)/(2m + 1)} > 2m$ for $0 < m < 1/2$. Letting
\[
f(m) \leqdef \log\bigg(\frac{2m + 1}{2m}\bigg)-\frac{2\log(2)}{2m + 1},
\]
the desired inequality is simply $f(m) > 0$ for $0 < m < 1/2$. We have that
\[
f'(m) = \frac{4m\log(2)-(2m + 1)}{m(2m + 1)^2} < \frac{2m-1}{m(2m + 1)^2} < 0, \qquad 0 < m < 1/2.
\]
Since $\lim_{m\uparrow1/2}f(m) = 0$, it follows that $f(m) > 0$ for all $0 < m < 1/2$. This completes the proof of inequality \eqref{mmm2}.

\vspace{3mm}

\noindent 3. Let $X\sim \mathrm{GH}(\lambda,\alpha,\beta,\delta,0)$. By \cite[Eq.\ (2.12)]{MR4141494} we have that $\mathrm{Mode}(X) < (\beta/\gamma^2)[\lambda-1/2 + \sqrt{(\lambda-1/2)^2 + \delta^2\gamma^2}]$ for $\lambda\in\R$ and $\delta > 0$. Combining this bound with inequality \eqref{eq:results.conj3.7.bound} yields inequality \eqref{mmm3}. The mean of the GH distribution is given by
\begin{equation}\label{mmean}\EE[X] = \frac{\delta\beta}{\gamma} \frac{K_{\lambda + 1}(\delta\gamma)}{K_{\lambda}(\delta\gamma)}
\end{equation}
(see \cite{bb80}). By \cite[Eq. (2.23)]{MR4141494} we have the bound
\[
\EE[X] > \frac{\beta}{\gamma^2}\Big[\lambda + \sqrt{\lambda^2 + \delta^2\gamma^2}\Big], \qquad \lambda\in\R,
\]
and combining this bound with inequality \eqref{eq:results.GH.upper.zero} yields inequality \eqref{mmm4} for $\lambda \geq 0$. The restriction that $\lambda \geq 0$ is a result of the fact that the upper bound for $\Med(X)$ in inequality \eqref{eq:results.GH.upper.zero} is only valid for $\lambda \geq 0$. We also note that an application of inequality \eqref{eq:Bessel.ratio.half} to equation \eqref{mmean} yields the bound
\[
\EE[X] \geq \frac{\beta}{\gamma^2}\Big[\lambda + 1/2 + \sqrt{(\lambda + 1/2)^2 + \delta^2\gamma^2}\Big], \qquad \lambda \leq -1/2.
\]
From this bound and inequality \eqref{eq:results.GH.upper.global} it follows that inequality \eqref{mmm4} also holds for $\lambda \leq -1/2$.

It remains to consider the case $-1/2 < \lambda < 0$. Set $D \leqdef \delta\gamma$ and
\[
\xi \leqdef \frac{\delta}{\gamma}\frac{K_{\lambda + 1}(D)}{K_{\lambda}(D)} = \frac{\EE[X]}{\beta}, \qquad a \leqdef \frac{\gamma^2\xi^2 - \delta^2}{2\xi}.
\]
Since $\lambda>-1/2$, we have that $K_{\lambda + 1}(D) > K_{\lambda}(D)$ (see \cite[Eq.\ (3.1)]{is91}). Consequently, $\xi > \delta/\gamma$ and $a > 0$. By the definition of $a$, $\gamma^2\xi^2 - 2a\xi - \delta^2 = 0$, and therefore, on taking the positive root,
\[
\xi = \frac{a + \sqrt{a^2 + D^2}}{\gamma^2} = \xi_a.
\]
Now set $\nu \leqdef \lambda - 1/2\in(-1,-1/2)$. For every $x > 0$,
\[
\frac{K_{\nu + 1}(x)}{K_{\nu}(x)} < 1 < \frac{a + \sqrt{a^2 + x^2}}{x},
\]
where the first inequality is due to \cite[Eq. (3.2)]{is91}, since $\nu<-1/2$.
Thus the reversed inequality in condition \eqref{eq:GH.comparison.lower.condition} of Lemma~\ref{lem:GH.comparison} holds for this choice of $a$, and so $b\mapsto\Psi_a(b)$ is strictly increasing on $(0,\infty)$. Since $\Psi_a$ is continuous at zero and $\beta > 0$, it follows that $\Psi_a(\beta) > \Psi_a(0) = 1/2$. Since $\xi = \xi_a$, the normal variance-mean mixture representation \eqref{eq:definitions.GH.mixture} and equation \eqref{mmean} give
\[
\PP(X \leq \EE[X]) = \PP(\beta W + \sqrt{W}N \leq \beta\xi_a) = \Psi_a(\beta) > \frac{1}{2}.
\]
The law of $X$ has a continuous strictly positive density, and therefore $\Med(X) < \EE[X]$. This proves inequality \eqref{mmm4} for $-1/2 < \lambda < 0$, and hence for all $\lambda\in\R$. Combining inequalities \eqref{mmm3} and \eqref{mmm4} proves \eqref{mmm5}. This completes the proof.

\section*{Funding}
\addcontentsline{toc}{section}{Funding}

Robert E.~Gaunt is funded by EPSRC grant EP/Y008650/1. Fr\'ed\'eric Ouimet is supported by the Natural Sciences and Engineering Research Council of Canada (NSERC) through Discovery Grant RGPIN-2026-04471 and Discovery Launch Supplement DGECR-2026-00449.

\section*{References}

\setlength{\bibsep}{0pt plus 0ex}


\begin{thebibliography}{99}
\addcontentsline{toc}{section}{References}

\bibitem{alm03} Alm, S. E. Monotonicity of the difference between median and mean of gamma distributions and of a related Ramanujan sequence. \emph{Bernoulli} $\mathbf{9}$ (2003), 351--371.

\bibitem{a05} Alzer, H. Proof of the Chen--Rubin conjecture. \emph{Proc. R. Soc. Edinb. A: Math.} $\mathbf{135}$ (2005), 677--688.

\bibitem{b26} Baricz, \'{A}., Prabhu, D. K., Singh, S. and Vijesh, V. A. Infinitely divisible modified Bessel distributions. \emph{Pac. J. Math.} $\mathbf{343}$ (2026), 261--313.

\bibitem{b77} Barndorff--Nielsen, O. E. Exponentially decreasing distributions for the logarithm of particle size. \emph{Proc. R. Soc. Lond. Ser. A} $\mathbf{353}$ (1977), 401--419.

\bibitem{bb80} Barndorff-Nielsen, O. E. and Bl\ae sild, P. \emph{Hyperbolic distributions and ramifications: Contributions to theory and application.} Research Report 68. Department of Theoretical Statistics, Institute of Mathematics, University of Aarhus, 1980.

\bibitem{bb85} Barndorff-Nielsen, O. E., Bl\ae sild, P., Jensen, J. L. and Sørensen, M. The fascination of sand. In Atkinson, A. C. and Fienberg, S. E. (eds), \emph{A Celebration of Statistics}, pp.\ 57--87. New York: Springer-Verlag, 1985.

\bibitem{bjs89} Barndorff-Nielsen, O. E., Jensen, J. L. and Sørensen, M. Wind shear and hyperbolic distributions. \emph{Bound.-Layer Meteorol.} $\mathbf{49}$ (1989), 417--431.

\bibitem{bks82} Barndorff-Nielsen, O. E., Kent, J. and Sørensen, M. Normal variance-mean mixtures and $z$ distributions. \emph{Int. Stat. Rev.} $\mathbf{50}$ (1982), 145--159.

\bibitem{bp06} Berg, C. and Pedersen, H. L. The Chen-Rubin conjecture in a continuous setting. \emph{Methods Appl. Anal.} $\mathbf{13}$ (2006), 63--88.

\bibitem{bs03} Bibby, B. M. and Sørensen, M. Hyperbolic processes in finance. In: Rachev, S. T. (Ed.), \emph{Handbook of Heavy Tailed Distributions in Finance}, Elsevier Science, Amsterdam, 2003, 211--248.

\bibitem{BockDiaconisHufferPerlman1987} Bock, M. E., Diaconis, P., Huffer, F. W. and Perlman, M. D. Inequalities for linear combinations of gamma random variables. \emph{Can. J. Stat.} $\mathbf{15}$ (1987), 387--395.

\bibitem{cr86} Chen, J. and Rubin, H. Bounds for the difference between median and mean of gamma and Poisson distributions. \emph{Stat. Probab. Lett.} $\mathbf{4}$ (1986), 281--283.

\bibitem{MR1195477} Choi, K. P. On the medians of gamma distributions and an equation of Ramanujan. \emph{Proc. Am. Math. Soc.} $\mathbf{121}$ (1994), 245--251.

\bibitem{craig} Craig, C. C. On the Frequency Function of $xy$. \emph{Ann. Math. Stat.} $\mathbf{7}$ (1936), 1--15.

\bibitem{d98} Dufresne, D. Algebraic properties of beta and gamma distributions, and applications. \emph{Adv. Appl. Math.} $\mathbf{20}$ (1998), 285--299.

\bibitem{eh04} Eberlein, E. and von Hammerstein, E. A. Generalized hyperbolic and inverse Gaussian distributions: limiting cases and approximation of processes, in: Dalang, R. C., Dozzi, M. and Russo, F. (Eds.), \emph{Seminar on Stochastic Analysis, Random Fields and Applications IV}, in: \emph{Prog. Probab.}, vol. 58, Birkh\"auser, Basel, 2004, 221--264.

\bibitem{ek95} Eberlein, E. and Keller, U. Hyperbolic distributions in finance. \emph{Bernoulli} $\mathbf{1}$ (1995), 281--299.

\bibitem{ep02} Eberlein, E. and Prause, K. The generalized hyperbolic model: financial derivatives and risk measures. In Geman, H., Madan, D., Pliska, S. R. and Vorst, T. (eds), \emph{Mathematical Finance--Bachelier Congress 2000: Selected Papers from the First World Congress of the Bachelier Finance Society, Paris, June 29--July 1, 2000}, 245--267. Berlin, Heidelberg: Springer, 2002.

\bibitem{FischerGauntSarantsev2025VGReview} Fischer, A., Gaunt, R. E. and Sarantsev, A. The Variance-Gamma Distribution: A Review. \emph{Statist. Sci.} $\mathbf{40}$ (2025), 235--258.

\bibitem{g14} Gaunt, R. E. Variance-Gamma approximation via Stein's method. \emph{Electron. J. Probab.} $\mathbf{19}$ (2014), no.\ 38, 1--33.

\bibitem{gaunt prod} Gaunt, R. E. A note on the distribution of the product of zero mean correlated normal random variables. \emph{Stat. Neerl.} $\mathbf{73}$ (2019), 176--179.

\bibitem{g22} Gaunt, R. E. The basic distributional theory for the product of zero mean correlated normal random variables. \emph{Stat. Neerl.} $\mathbf{76}$ (2022), 450--470.

\bibitem{MR4141494} Gaunt, R. E. and Merkle, M. On bounds for the mode and median of the generalized hyperbolic and related distributions. \emph{J. Math. Anal. Appl.} $\mathbf{493}$ (2021), Art.\ 124508.

\bibitem{pog1} G\'orska, K., Horzela, A., Jankov Ma\v{s}irevi\'c, D. and Pog\'any, T. K. Observations on the McKay $I_\nu$ Bessel distribution. \emph{J. Math. Anal. Appl.} $\mathbf{516}$ (2022), Art.\ 126481.

\bibitem{gm77} Groeneveld, R. A. and Meeden, G. The Mode, Median, and Mean Inequality. \emph{Am. Stat.} $\mathbf{31}$ (1977), 120--121.

\bibitem{HolmAlouini2004} Holm, H. and Alouini, M.-S. Sum and difference of two squared correlated Nakagami variates in connection with the McKay distribution. \emph{IEEE Trans. Commun.} $\mathbf{52}$ (2004), 1367--1376.

\bibitem{is91} Ifantis, E. K. and Siafarikas, P. D. Bounds for modified Bessel functions. \emph{Rend. Circ. Mat. Palermo} $\mathbf{40}$ (1991), 347--356.

\bibitem{pog2} Jankov Ma\v{s}irevi\'c, D., Pog\'any, T. K. and Uji\'c, N. Observations on the McKay $I_\nu$ Bessel distribution II. \emph{J. Math. Anal. Appl.} $\mathbf{550}$ (2025), Art.\ 129569.

\bibitem{kkp01} Kotz, S., Kozubowski, T. J. and Podg\'{o}rski, K. \emph{The Laplace Distribution and Generalizations: A Revisit with Applications to Communications, Economics, Engineering, and Finance.} Birkh\"auser, Boston, 2001.

\bibitem{ln10} Laforgia, A. and Natalini, P. Some inequalities for modified Bessel functions. \emph{J. Inequal. Appl.} (2010), Art.\ ID 253035, 10 pp.

\bibitem{Lyon2021} Lyon, R. F. On closed-form tight bounds and approximations for the median of a gamma distribution. \emph{PLoS ONE} $\mathbf{16}$ (2021), e0251626.

\bibitem{Lyon2023} Lyon, R. F. Tight bounds for the median of a gamma distribution. \emph{PLoS ONE} $\mathbf{18}$ (2023), e0288601.

\bibitem{mcc98} Madan, D. B., Carr, P. P. and Chang, E. C. The variance gamma process and option pricing. \emph{Eur. Finance Rev.} $\mathbf{2}$ (1998), 79--105.

\bibitem{ms90} Madan, D. B. and Seneta, E. The variance gamma (V.G.) model for share market returns. \emph{J. Bus.} $\mathbf{63}$ (1990), 511--524.

\bibitem{mo09} Marshall, A. W., Olkin, I. and Arnold, B. C. \emph{Inequalities: Theory of Majorization and Its Applications}, second edition, Springer, New York, 2011.

\bibitem{m32} McKay, A. T. A Bessel function distribution. \emph{Biometrika} $\mathbf{24}$ (1932), 39--44.

\bibitem{mfe05} McNeil, A. J., Frey, R. and Embrechts, P. \emph{Quantitative Risk Management: Concepts, Techniques, and Tools.} Princeton University Press, Princeton, 2005.

\bibitem{np16} Nadarajah, S. and Pog\'{a}ny, T. K. On the distribution of the product of correlated normal random variables. \emph{C. R. Acad. Sci. Paris, Ser. I} $\mathbf{354}$ (2016), 201--204.

\bibitem{NISTDLMF} Olver, F. W. J., Lozier, D. W., Boisvert, R. F. and Clark, C. W., eds. \emph{NIST Handbook of Mathematical Functions}. Cambridge University Press, New York, 2010.

\bibitem{rs16} Ruiz-Antol\'in, D. and Segura, J. A new type of sharp bounds for ratios of modified Bessel functions. \emph{J. Math. Anal. Appl.} $\mathbf{443}$ (2016), 1232--1246.

\bibitem{s11} Segura, J. Bounds for ratios of modified Bessel functions and associated Tur\'an-type inequalities. \emph{J. Math. Anal. Appl.} $\mathbf{374}$ (2011), 516--528.

\bibitem{s23} Segura, J. Simple bounds with best possible accuracy for ratios of modified Bessel functions. \emph{J. Math. Anal. Appl.} $\mathbf{526}$ (2023), Art.\ 127211.

\bibitem{vdv98} van der Vaart, A. W. \emph{Asymptotic Statistics}. Cambridge University Press, Cambridge, 1998.

\bibitem{v78} van Zwet, W. R. Mean, median, mode II. \emph{Stat. Neerl.} $\mathbf{33}$ (1979), 1--5.

\bibitem{wb32} Wishart, J. and Bartlett, M. S. The distribution of second order moment statistics in a normal system. \emph{Proc. Cambridge Philos. Soc.} $\mathbf{28}$ (1932), 455--459.

\end{thebibliography}
\end{document}